\documentclass[11pt]{article} 
\usepackage[margin=0.8in]{geometry}
\usepackage{amsmath,amsfonts,amssymb,amsmath,amscd}
\usepackage[active]{srcltx}
\usepackage{graphicx,array}
\usepackage{amsthm}
\usepackage{amsbsy}
\usepackage{tikz}
\usepackage{soul}
\usepackage{multirow}

\usepackage{float}
\usepackage{enumerate}

\numberwithin{equation}{section}

\renewcommand{\thefigure}{\thefigure.\arabic{equation}}
\numberwithin{figure}{section}

\usepackage[normalem]{ulem}

\usetikzlibrary{graphs}
\usepackage{tkz-berge}
\usepackage{tikz}
\usetikzlibrary{positioning}

\usepackage{float}
\usepackage{enumerate}

\usepackage[latin1]{inputenc}
\usepackage{tikz}
\usetikzlibrary{positioning}

\usepackage[english]{babel}
\usepackage{amsmath}
\usepackage{amssymb}
\usepackage{amsthm}
\usepackage{graphicx}
\usepackage{ulem}

\theoremstyle{plain}
\newtheorem{teo}{Theorem}[section]
\newtheorem{prop}[teo]{Proposition}

\newtheorem{coro}[teo]{Corollary}

\newtheorem{defin}[teo]{Definition}

\newtheorem{obs}[teo]{Remark}

\newcommand{\Rr}{\mathbb{R}}
\newcommand{\CG}{\mathcal{G}}
\newcommand{\CC}{\mathcal{C}}
\newcommand{\CE}{\mathcal{E}}

\newcommand{\CT}{\mathcal{T}}
\newcommand{\CH}{\mathcal{H}}
\newcommand{\CF}{\mathcal{F}}

\newcommand{\tq}{\, | \,}
\newcommand{\CL}{\mathcal{L}}

\begin{document}
\vspace{-5mm}

\title{The realization of admissible graphs for coupled vector fields}
\date{}
\maketitle
\vspace*{-1cm}
\centerline{\scshape Tiago de Albuquerque Amorim\footnote{Email address: tiagoamorim8@usp.br
 }}
{\footnotesize \centerline{Mathematics Department, ICMC}
\centerline{University of S\~ao Paulo} \centerline{13560-970 P.O. box
668, S\~ao Carlos, SP - Brazil }}
\vspace*{0,5cm}
\centerline{\scshape Miriam Manoel\footnote{Email address:
miriam@icmc.usp.br (corresponding author)}}
{\footnotesize \centerline{Mathematics Department, ICMC}
\centerline{University of S\~ao Paulo} \centerline{13560-970 P.O. box 668, S\~ao Carlos, SP - Brazil }}

\vspace*{1cm}

\begin{abstract}
In a coupled network cells can interact in several
ways. There is a vast literature from the last twenty years that investigates this interacting
dynamics under a graph theory formalism, namely as a graph endowed with an input-equivalence relation
on the set of vertices that enables a characterization of the admissible vector fields that rules the network dynamics. The present work goes in the direction of answering an inverse problem: for $n \geq 2$, any mapping on $\Rr^n$
can be realized as an admissible vector field for some graph with the number of vertices depending on (but not necessarily equal to) $n$. Given a mapping, we present a procedure to construct all non-equivalent admissible graphs, up to the appropriate equivalence relation. We also give an upper bound for the number of such graphs. As a consequence, invariant subspaces under the vector field can be investigated as the locus of synchrony states supported by an admissible graph, in the sense that a suitable graph can be chosen to realize couplings with more (or less) synchrony than another graph admissible to the same vector field. The
approach provides in particular a systematic investigation of occurrence of chimera states in a network of van der Pol identical oscillators.
\\ {$\,$}  

\end{abstract}

\section{Introduction} \label{sec:introduction}
Studies involving several dynamical systems and their interaction have been carried out for a long time. We go back to the 80's to recall physical, biological, mechanical systems that have been interpreted in many ways as coupled cell systems;
for example, Josephson junction arrays \cite{[1988]}, semiconductor coupled lasers or multimode solid state laser systems in \cite{WW1988} and \cite{Braci1990},  central pattern generators and symmetric chains of weakly coupled oscillators in  \cite{Kopell1986}, \cite{Kopell1988} and \cite{Kopell1990}, sympatric speciation \cite{StewartCohen2000}, normal mode vibrations of a loaded string and linear motion of a triatomic molecule \cite{Fowles1986}, the classical $n$-body dynamics  \cite{Griffiths1985}, among many others. 
 It was  around 2002 that the authors M. Golubitsky, I. Stewart and collaborators (\cite{GNS}, \cite{SM2002}, \cite{SMT}, \cite{SMP}) started to formulate the notion of a \textit{coupled cell network}, a rich systematic way to study coupled dynamics under a graph theory formalism \cite{SM2002}, establishing a general settings for simple graphs (no multiarrows or loops)  \cite{SMP} and for multigraphs (possible multiple arrows and loops) \cite{SMT}. Under that formulation, a network graph is more than a finite set of vertices with a finite set of arrows, for there are distinct types of vertices and arrows to be taken into consideration to represent abstractly a system of ODEs equipped with interacting individual cells as canonical observables. Each vertex represents an individual cell which is governed by an autonomous system of ODEs, and the set of edges encodes couplings, so that a well defined `admissible' vector field  is assigned to this network graph. 

A central matter is to know to what extent the rigidity of the network graph topology constrains the investigation of the associated dynamics. It is well known that topologically distinct network graphs can lead to the same set of admissible vector fields (see \cite{SMT}). Hence, as naturally expected, this formalism is largely based on different kinds of equivalence relations, as we briefly mention in what follows. One key identification of the cells inside a network is the notion of `input isomorphism' \cite{SMP}, under which two cells are equivalent  if the dynamics of the cells are governed by the same differential equations, up to a permutation of the variables. Another identification is defined by the `balanced equivalence relation' between cells, which stratifies the set of cells in terms of equalities among cell variables representing a synchrony state of the network. This is a major topic in the investigation of coupled dynamics, with a countless number of works devoted to it.  The authors in \cite{AguiarDiasAlg}  present an algorithm that generates the lattice of synchrony subspaces of a given network graph. Balanced equivalence relations are also a strong reason for the setting of the multigraph formalism, that is, for graphs with multiple arrows of same type connecting two vertices, since the associated quotient graphs -- those that encode the synchronies in the network -- are generally multigraphs. Alongside these relations, much attention has also been driven to relations between networks, namely the `automorphisms' of a network graph (formalized in \cite{AntoneliStewart}) and the `ODE-equivalence' between networks (see  \cite{DiasStewart} and \cite{SMP} for example). Two nonisomorphic networks can have equivalent dynamical behavior, and this is detected by ODE-equivalence. 

If we now look at the usual way of modelling a coupled dynamics through an associated vector field, relevant distinctions emerge in assigning a simple graph (one edge of a type connecting two vertices), or a multigraph, that realizes it as its admissible vector field. It is expected that this assignment is not unique -- just take two distinct network graphs inside an ODE-equivalence class, possibly one chosen to be a simple graph and the other to be a multigraph. 
However, a more subtle case is equally possible, as it is shown with the elementary example below, also illustrating that, although it may be simpler to deal with simple graphs, multiple arrows are sometimes the appropriate way to model certain couplings.  Consider the following differential equations on $\Rr^3$, 
\begin{equation} \label{eq:3-cell}
\begin{array}{lll}
\dot{x}_1 & = & x_1+x_1^3  \\
\dot{x}_2 & = & x_2+x_2^2x_3  \\
\dot{x}_3 & = & x_3+x_1x_2x_3. \\
\end{array}
\end{equation}
These equations can model a 3-cell network represented by the simple directed graph of Fig.~\ref{graph1} (left), which is an inhomogeneous graph with cells with distinct valencies. However, the same equations can also model a network represented by the multigraph of Fig.\ref{graph1} (right), which is a regular graph, with all cells of same type and same input and all arrows of same type, which therefore presents synchronies. In Subsection~\ref{subseq: simple graph versus multigraph} we return to this example. Now, if, on one hand, from a given graph and a predetermined set of cell domain  there is a unique general form of an admissible vector field, the inverse problem, on the other hand, is not uniquely solved, as the example above shows. For this elementary example, the possible synchronous configurations for both graphs coincide, namely total synchrony and total asynchrony, but this is not the case in general, as we shall see in many examples here. 

\begin{figure}[h!]
	\centering 
	\includegraphics[width=5cm]{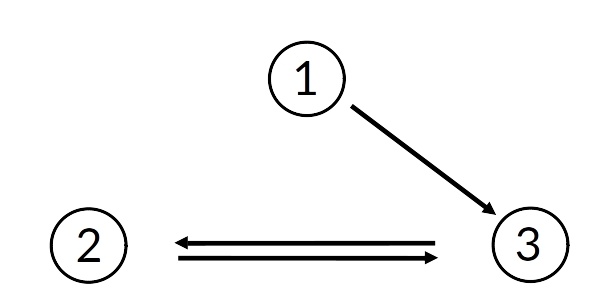} \hspace*{2cm}  \includegraphics[width=5cm]{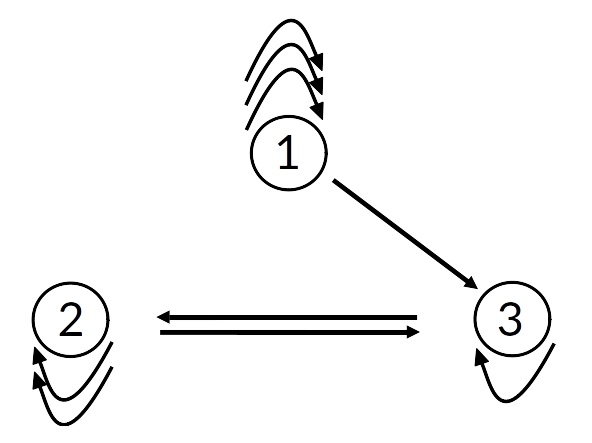} 
	\caption{\small An admissible nonregular simple graph (on the left) and an admissible regular  multigraph (on the right) for the system of equations (\ref{eq:3-cell}).}
	\label{graph1}
\end{figure}

In distinct areas of applied science, the possible distinct couplings that can be arranged from a given vector field modeling an interacting dynamical system  is a relevant problem to be understood. The aim here is to discuss and solve this inverse problem in general. For a given vector field, we construct non ODE-equivalent graphs that realize it as their admissible vector field, which we shall call `admissible graphs'. 

In order for a vector field to be a candidate to model a coupled dynamics with nontrivial coupling, the starting point is to choose the number of cells as well as the dimensions of the cell variables. Next, we look at symmetries, namely invariance of its component functions under permutations of (sets of) variables; for example, if the components of the vector field are invariant under all the permutations of its variables, then there is only one resulting admissible graph, which is a homogeneous graph. Symmetries are then related to the assignment of {\it generating functions} to the components of the vector field and, as we shall see, this is closely related to the number of distinct admissible graphs. We remark that  there are special cases for which two distinct such choices may lead to network graphs whose synchronies can be related, in the sense that one is deduced from the other (see Section \ref{sec:van der Pol}).

As already mentioned above, isomorphic graphs and more generally ODE-equivalent graphs must be identified in the search of distinct types of dynamics. Our procedure takes these identifications into consideration, producing the complete list of non-isomorphic non ODE-equivalent graphs that can be realized from a given vector field. We also give an upper bound for the number of elements in this list (Proposition~\ref{propvartheta}). 

One particular motivation to carry out with this procedure relies on the approach of a many works to study  `chimera' states of  a vector field (Section~\ref{sec:van der Pol}). We cite \cite{AS}, \cite{KB}, \cite{PA}, \cite{MTFH}, \cite{Cerdeira}  for example, and \cite{Anna} for a timely overview on the subject. A chimera state is a peculiar partial synchronization patterns in networks,  defined as a spatio-temporal pattern in which a system of
identical oscillators is split into coexisting regions of coherent (phase and frequency locked) and incoherent (drifting) oscillation \cite{AS}. Hence, it is a state in which the  whole set of cells is broken into a synchronous part and
an asynchronous part. This is a phenomenon in interacting populations of intense research efforts in physical, biological, chemical and social systems \cite{ABCR}. We notice that although a well-known phenomenon in non-identical coupled oscillators, the coexistence of coherence and incoherence was observed by Kuramoto and Battogtokh for a system of identical oscillators \cite{KB}; we refer to \cite{PA} for a comprehensive list of references about this subject. In these works, the research is carried out from a model given by a vector field, based on numerical or analytical investigation of the set of the proposed equations. For example, the authors in \cite{MTFH} carry out numerical simulations of a mechanical experiment with two subpopulations of identical metronomes on two coupled swings with chimera states emerging robustly. In this experiment, we observe that the modeling equations admit an invariance under the permutation of the two subsets of metronomes that has not been taken into account apparently, and this might reveal distinct chimera states. Our procedure is in particular a way to describe chimeras from the possible synchronized states that appear from the graphs admissible from the given vector field. In this setting, the analysis of robust attracting chimeras should correspond to the numerically observed chimeras, and this is an ongoing investigation so far. Section~\ref{sec:van der Pol} is addressed to expand on this topic in more detail with a concrete example. \\

The structure of this paper is as follows. In Section~\ref{sec:preliminaries} we mostly review known results and standard notation from the literature of coupled networks.  
Section~\ref{sec:realization} is the central section, addressed to present the construction procedure  of admissible graphs of a given $C^1$ vector field (Subsection~\ref{subseq:method}), an illustrative example (Subsection~\ref{subsec:vf on R4}) and a brief discussion about a possible resulting simple graph or multigraph for a given vector field 
(Subsection~\ref{subseq: simple graph versus multigraph}).  The relation among admissible graphs of a smooth vector field obtained from our procedure up to isomorphism and ODE-equivalence is given in a set of results in Section~\ref{sec:realization}, which are our main results.  Section~\ref{sec:van der Pol} is an application of the results, where we discover chimera states in a network of six van der Pol oscillators  as particular synchronous configurations of some of the admissible graphs.

\section{Coupled cell networks} \label{sec:preliminaries}

We leave to this section some basic results, as well as the definitions, notation and results from the literature, that we use in the forthcoming sections.

Cells are individual dynamical systems represented by vertices of a graph whose interactions are represented by the edges. A {\it coupled cell network} $\CG$, with possible multiple couplings and self-couplings, consists of a finite set of cells $\CC=\{1,\ldots,n\}$ with an equivalence relation $\sim_{\CC}$; 
a finite set of arrows (or edges) $\CE$ with a equivalence relation $\sim_{\CE}$;
two maps $\CH,\CT \colon \CE \rightarrow \CC$, where, for $e \in \CE$,  $\CH(e)$ and  $\CT(e)$ are the head and the tail of $e$, with a compatibility condition,
\begin{eqnarray} e_1,e_2 \in \CE, e_1 \sim_{\CE} e_2 \quad \Rightarrow  \quad  \CH(e_1) \sim_{\CC} \CH(e_2),  \ \CT(e_1) \sim_{\CC} \CT(e_2). \nonumber 
\end{eqnarray}
We shall also denote $\CG=(\CC,\CE, \sim_{\CC},\sim_{\CE})$. From now on, we shall also refer to a coupled cell network as a network graph, or simply as a graph. A self-coupling is an edge $e$ such that  $\CH(e)=\CT(e)$ and multiarrows are distinct edges $e_1, e_2$ such that  $\CH(e_1) =\CH(e_2)$ and  $\CT(e_1) = \CT(e_2)$.  For a graph $\CG$  with $n$ cells, the order-$n$ adjacency matrix of  each type of edges $\xi \in \CE/{\sim_{\CE}}$ shall be denoted by $A_{\CG}^{\xi}$ 
where 
\begin{eqnarray} (A_{\CG}^{\xi})_{ij}= | ( \CT^{-1}(j) \cap \CH^{-1}(i) \cap \xi)| ; \nonumber  
\end{eqnarray}
when it is clear from the context, we shall omit $\CG$ from this notation. For $c \in \CC$, the {\it input set} of $c$ is 
$$I(c)=\{e \in \CE : \  \CH(e)=c\}=\CH^{-1}(c).$$ The input equivalence relation $c \sim_I d$ between two cells {$c, d \in \CC$} is given by the existence of an arrow-type preserving bijection 
\begin{eqnarray} 
\beta \colon I(c) \rightarrow  I(d), \nonumber 
\end{eqnarray} 
that is, $\beta(e)\sim_{\CE} e,$ for all $e \in I(c).$ The set of the input isomorphisms $\beta$  is denoted by $B(c, d)$. The union of all $B(c,d)$ has a groupoid structure with respect to the composition. If $c \sim_I d$ and $I(c)$ is nonempty, then from the compatibility condition we have $c \sim_{\CC} d$. Hence, if $I(c)$ is empty, then we require that $c \sim_{I} d$ implies  $c \sim_{\CC} d$.

A network graph $\CG$ is  {\it homogeneous} if all cells are $I$-equivalent, in which case $B(c,d)  \neq \emptyset,$  for all $c,d \in \CC$; and $\CG$ is {\it regular} if it is homogeneous with one type of edge.  

An isomorphism between two network graphs $\CG_1$, $\CG_2$ is given in the natural way:  if there exists a cell bijection $\gamma_{\CC} \colon \CC_1 \rightarrow \CC_2$ and an arrow bijection $\gamma_{\CE} \colon \CE_1 \rightarrow \CE_2$ such that 
	
$$		c \sim_{\CC_1} d  \Leftrightarrow  \gamma_{\CC}(c) \sim_{\CC_2} \gamma_{\CC}(d) \nonumber$$ 
$$ e \sim_{\CE_1} e' \Leftrightarrow \gamma_{\CE}(e) \sim_{\CE_2} \gamma_{\CE}(e')$$
$$\CT_2(\gamma_{\CE}(e))=\gamma_{\CC}(\CT_1(e)),   \CH_2(\gamma_{\CE}(e))=\gamma_{\CC}(\CH_1(e)), \  \forall e \in \CE_1.$$
Hence, $\CG_1$ and $\CG_2$ are isomorphic if, and only if, by a rearrangement of cells the adjacency matrices of $\CG_1$ and $\CG_2$ are the same. We denote by $Iso(\CG)$ the set of isomorphisms of $\CG$ into itself (see Corollary \ref{cor: Iso}).

\begin{obs}
{\rm In a simple network graph (no multiple arrows), each arrow $e \in \CE$ can be identified with the ordered pair $(\CT(e),\CH(e))$. In a network graph with no loops, for each arrow we have $\CT(e) \neq \CH(e)$. Hence, in the simple graph formalism,  $\CE$ can be viewed as a subset of $\CC \times \CC$ with no elements of the form $(c,c), \, \, c \in \CC$ and the input set of $c \in \CC$ can be viewed as the subset of $\CC$ given by  $ \{d \in \CC \tq (d,c) \in \CE\}.$ Let us just point out that in \cite{SMP} it is assumed that $\{(c,c) \tq c \in \CC\} \subset \CE$ together with the condition that $(c,c) \sim_{\CE} (d,d')$ holds only  if $d'=d$ and $c \sim_{\CC} d$, aiming to have in hand the useful condition that $c \in I(c)$. However, for formal purposes, this should be avoided, because an internal edge $(c,c)$ can not be related to an external edge $(d,d')$, $d \neq d'$.}
\end{obs}

\subsection{Admissible vector fields} 

For each cell $c\in \CC$, let $P_c$ denote its domain, or the so-called \textit{cell phase space}, which is a nontrivial finite-dimensional real vector space. We require 
\begin{eqnarray}c \sim_{\CC} d \Rightarrow
P_c = P_d, \label{cap1eq2},
\end{eqnarray}
in which case we use the same coordinate system for $P_c$ and $P_d$. The \textit{total phase space} and the \textit{coupled phase space} of $c \in \CC$ are, respectively, 
$$P= \prod _{c \in \CC} P_c \ ,  \ \  P_{I(c)}=\prod _{e \in I(c)} P_{\CT (e)},$$ 
the coordinate system of the later being ${y}=(y_e)_{e \in I(c)}$. 
Consider now the map 
$$\begin{array}{cccc}
\pi_{I(c)} \colon & P & \longrightarrow & P_{I(c)} \\
& x = (x_{c'})_{c' \in \CC} & \longmapsto & \pi_{I(c)}(x)=(x_{\CT(e)})
\end{array},$$ 
which repeats each variable $x_{c'}$ as many as the number of edges from $c'$ to $c$ in $I(c)$ . In particular, from \eqref{cap1eq2} it follows that,  if $c \sim_I d$, then $P_{I(c)}=P_{I(d)}$. 

Finally, for any $c \in \CC,$ the vertex group $B(c, c)$  has a natural action  on  $P_{I(c)}$ as follows: for each $\beta \in B(c,c)$ we have  
$$\beta^{*} \colon P_{I(c)} \to    P_{I(c)}$$
defined by
$$\beta^{*}(y)_e= y_{\beta(e)}, \quad \forall e \in I(c).$$ 
Similarly, for $\beta \in B(c,d)$,  we also define $\beta^{*} \colon P_{I(d)} \to    P_{I(c)}$.  We can now recall the concept of admissible vector field:

\begin{defin} \label{definAVF}
	A vector field $f \colon P \to P$ is $\CG$-admissible if the following hold:
	\begin{itemize}
		\item [(a)] Domain condition: For all $c \in \CC$, the component $f_c$  depends only on the internal phase space; that is, there exists $\hat{f}_c \colon P_c \times P_{\CT(I(c))} \rightarrow P_c$ such that
		\begin{eqnarray}
		f_c(x)=\hat{f}_c(x_c, \pi_{I(c)}(x)). \label{GST32} 
		\end{eqnarray}
		\item [(b)] Equivariance condition: For any pair $c,d \in \CC$ and $\beta \in 
		B(c,d)$,
		\begin{eqnarray}
		\hat{f}_d(x_d, {y})=\hat{f}_c(x_d, \beta^*(y)) \quad \forall (x_d, {y}) \in P_d \times P_{I(d)}. \label{fbeta}
		\end{eqnarray}
	\end{itemize}
\end{defin}
We notice that when $c=d$ the equivariance condition means that $\hat{f}_c$ is $B(c,c)$-invariant in the $y$-variable. In this case, we write
\begin{eqnarray}\hat{f}_c(x_c, \overline{y^1_1, \ldots, y^1_{|I(c) \cap \xi_1|}}, \ldots,\overline{y^r_1, \ldots, y^r_{|I(c) \cap \xi_r|}}), \label{fbar}
\end{eqnarray}
where over bar means invariance under the permutations of these variables. \\

We denote by $\CF({\CG};P)$ the set of smooth $\CG$-admissible  vector fields on $P$, and by $\CL({\CG};P)$ the set of linear $\CG$-admissible vector fields on $P$. We finish this subsection with the following characterization, which is straightforward and shall be very useful in the sequel:
\begin{eqnarray}
	\CL(\CG;\Rr^n)= \bigl\{ diag(t_1, \ldots,t_n)+\sum_{\xi \in \CE/{\sim_{\CE}}} diag(t_1^{\xi}, \ldots, t_n^{\xi})A_{\CG}^{\xi} \  : \ i \sim_I j \Rightarrow t_i=t_j, \, t_i^{\xi}=t_j^{\xi}  \bigr\}.\, \,  \, \, \, \label{LGRn}
	\end{eqnarray}

\subsection{ODE-equivalence}

Topologically distinct network graphs can determine the same space of smooth admissible vector fields. This statement is precisely the notion of \textit{ODE-equivalence} between network graphs (see \cite{DiasStewart}). It turns out that verifying the ODE-equivalence can be reduced to the linear level. In fact,  as established in \cite[Theorem 5.1, Corollary 7.7]{DiasStewart},  two network graphs $\CG_1$ and $\CG_2$ of $n$ cells are ODE-equivalent if, and only if, there exists an input-preserving bijection $\gamma \colon \CC_1 \rightarrow \CC_2$ such that
\begin{eqnarray} \label{eq:ode equivalence}
\CL(\CG_1,\Rr^n)&=&\gamma^T \CL(\CG_2,\Rr^n) \gamma \label{eq: gama ODE}.
\end{eqnarray}

It is direct from \eqref{LGRn}  that if $\CG_1$ and $\CG_2$ are isomorphic, then they are ODE-equivalent. The converse holds for the homogeneous simple case:

\begin{prop} \label{prop:homog simple graphs}
Let $\CG_1$ and $\CG_2$ be homogeneous simple network graphs. If $\CG_1$ and $\CG_2$ are ODE-equivalent, then they are isomorphic. 
\end{prop}

\begin{proof}
Let $\gamma\colon \CC_1 \rightarrow \CC_2$  be a bijection that realizes the ODE-equivalence (\ref{eq:ode equivalence}). 
Since each graph is simple, if a position of any of its adjacency matrix is 1, then the same position in all the other of its adjacency matrices is zero. In particular, the adjacency matrices of a simple graph are linearly independent. Hence, from \eqref{LGRn}, $\CG_1$ and $\CG_2$ have the same number of adjacency matrices. For simplicity, suppose this number is $2$ and  let $A_{\CG_1}^k $ and $A_{\CG_2}^k $, $k = 1,2,$ be their adjacency matrices.

Let $s^1,s^2$ be such that 
$$ A_{\CG_1}^1=s^1(\gamma ^T A_{\CG_2}^1\gamma)+s^2(\gamma^T A_{\CG_2}^2\gamma).$$
Now, let $(i,j)$ be such that $(A_{\CG_1}^1)_{ij}=1$. Without lost of generality, we can assume that 
\begin{eqnarray} (\gamma ^T A_{\CG_2}^1\gamma)_{ij}=1 \mbox{ and }(\gamma ^T A_{\CG_2}^2\gamma)_{ij}=0. \nonumber \end{eqnarray}
Hence, 
\begin{eqnarray} 
A_{\CG_1}^1=(\gamma ^T A_{\CG_2}^1\gamma)+s^2(\gamma^T A_{\CG_2}^2\gamma). \label{eq1} 
\end{eqnarray}
But $(A_{\CG_1}^2)_{ij}=0$. It then follows that 
$$ A_{\CG_1}^2=(\gamma^T A_{\CG_2}^2\gamma),$$
and then $s^2=0$ in \eqref{eq1}. So the result follows.  
\end{proof}
The proposition above is very useful to deduce the complete list of distinct non equivalent admissible graphs of a given vector field for the homogeneous simple case. See the case study of Subsection~\ref{sec: vf on R6}.  

\subsection{Synchrony} \label{subsec:ber}
For $f \in \CF({\CG};P)$,  a
synchrony in the coupled cell system 
$$ \dot{x}=f(x)$$
occurs if two or more cells of a solution $x(t)$ behave identically, that is, if $c$ and $d$ are any two of these cells, then 
$$x_c(t)=x_d(t), \, \forall t.$$
The authors in \cite{SMT, SMP} characterize occurrence of synchronies in the dynamics from the graph architecture. Here we briefly present the general idea and the main result: Let $\bowtie$ be an equivalence relation on $\CC$. Then 
\begin{itemize}
	\item  $\bowtie$ is {\it balanced} if $c \bowtie d$ implies that there exists  $\beta \in B(c,d)$ such that $\CT(e) \bowtie \CT(\beta(e)), \, \forall e \in I(c)$; 
	\item $\bowtie$ is {\it robustly polysynchronous} if for any choice of total phase space $P$ we have $f(\Delta_{\bowtie}) \subset \Delta_{\bowtie}$, 
 for all $f \in \CF(\CG;P)$, where $$\Delta_{\bowtie}=\{x \in P \tq x_c=x_d \Leftrightarrow c \bowtie d\}. $$
\end{itemize}
Such a polydiagonal subspace is called a {\it synchrony subspace}. 
Notice that the first notion above is related directly to the graph $\CG$, whereas the second  is related to the class $\CF(\CG;P)$, but it turns out that both notions are equivalent, which has been  established in \cite[Theorem 4.3]{SMT}. 
Moreover, these are also equivalent to
\begin{eqnarray} A_{\CG}^{\xi}(\Delta_{\bowtie}) \subset \Delta_{\bowtie}, \, \, \, \forall \xi \in \CE/_{\sim_{\CE}}. \label{eq: condition to balanced} 
\end{eqnarray} 
In words,  synchrony subspaces are invariant  under all adjacency matrices of $\CG$.

\begin{obs} \label{obs code} {Based on (\ref{eq: condition to balanced}), the authors in \cite{AguiarDiasAlg} develop an algorithm to return all the possible synchronies for a given network graph. The main idea is to look for polydiagonal invariant subspaces from the eigenvectors of the Jordan decomposition of the adjacency matrix. We have implemented the algorithm with a Mathematica code  and the data in Table~1 (see Section~\ref{sec:van der Pol}) and Table~2 (see Subsection~\ref{sec: vf on R6}) are deduced with an extensive use of this code.}
\end{obs}

We end this subsection recalling that there are synchronies inherited by the symmetries of the network graph: for the group ${\rm Aut}(\CG)$ of automorphisms of $\CG$ (see \cite[Definition~3.1]{AntoneliStewart}), the fixed-point subspace of any subgroup $\Sigma <  {\rm Aut}(\CG)$  is a robust polydiagonal subspace $\Delta_{\bowtie}$ of $\Rr^n$, where $\bowtie$ is defined through $\Sigma$ \cite[Proposition~3.3]{AntoneliStewart}. There are particular classes of network graphs for which all the balanced equivalence relations are precisely the ones defined in this way. This is the case of the network graph ${\rm G}_6$ given by the regular ring with nearest and next nearest neighbor coupling. This network graph is the admissible graph of the coupling of van der Pol oscillators of Section~\ref{sec:van der Pol}. We leave to that section the details for this graph synchronies.

\section{The realization of the admissible graphs} \label{sec:realization}
As recalled in Section~\ref{sec:preliminaries}, based on the action of a groupoid of symmetries of a given network graph, the authors in \cite{SMT} and \cite{SMP}  formulate in algebraic terms the class of admissible vector fields on the total phase space that are `compatible' with the labeled structure of a given graph.  In this section we follow that formulation to study the problem in the inverse direction. We give the procedure to construct the network graphs associated with a given vector field, namely the admissible graphs for this vector field. \\

Before that, we point out that permutations play a major role in the two directions. Let us illustrate the two approaches together, with the elementary graphs of Fig.~\ref{fig:simplest}. The general admissible vector field for the network graph on the left is of the form
\[ \begin{array}{ll}
\dot{x_1} = & f(x_1, x_2, x_3) \\
\dot{x_2} = & g(x_2, x_1) \\
\dot{x_3} = & h(x_3, x_1), \\
\end{array} \]
for any three-variable function $f$ and two-variable functions $g$ and $h$. But it should be reasonable to go on with the analysis assuming an additional necessary condition, namely $f$ non invariant under the permutation of $x_2$ and $x_3$ (together with $g$ and $h$ distinct), in the same way that this permutation invariance is a necessary condition for the above vector field to be admissible for the network graph on the right (together with $g = h$). For the inverse problem, the possible (and not possible) permutation invariances can be taken into account in the initial process of constructing the graphs, soon after the choice of the number of vertices. In fact, this is the basis for the stepwise procedure of Subsection~\ref{subseq:method} to produce admissible graphs, as well as `optimized' admissible graphs as a final step.  By an optimized admissible graph we mean a graph that indeed depicts the permutation invariances of the components of this vector field with the least number of edge types. In the broad sense of admissibility, any vector field yields an admissible graph, for we can take the complete graph with all edges in  distinct classes for example. However, this graph has no nontrivial symmetries and should not be of much interest if we are to model many features of coupled dynamics such as synchronization.

\begin{figure}[h!]
	\centering %
	\includegraphics[width=5cm]{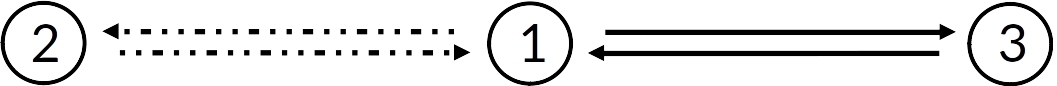} \quad \quad \quad \quad  \includegraphics[width=5cm]{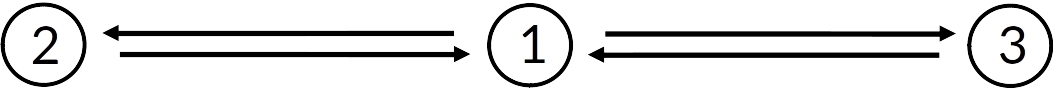} 
	\caption{\small (a) A 3-cell graph with distinct couplings and (b) a 3-cell graph with identical couplings. }
	\label{fig:simplest}
\end{figure}

\subsection{The stepwise procedure} \label{subseq:method}

For a given $C^1$ vector field, this is a step-by-step procedure to give the complete list of admissible graphs, up to ODE-equivalence. It is also an optimization method of choosing  systematically input values  from within an allowed set from the mapping components. The general idea is: compare the components according to \eqref{fbeta}; by comparison,  define the equivalence relation $\sim_{\CC}$ of vertices and also the equivalence relations $\sim_k$ of edges in four steps, $k=1,2,3,4$. The relation   $\sim_{\CC}$ is established in step 1. For $k = 2, 3, 4$, the relation $\sim_k$ is constructed in the $k^{\rm th}$ step and {its classes are obtained by joining classes of $\sim_{k-1}$.} \\

Let $f \in C^1(\Rr^n;\Rr^n)$ be a given vector field.  \\

\noindent \textbf{Step 1.} Choose the number of cells $n_0$, $\CC = \{1, \ldots, n_0\}$. The natural choice is $n_0 = n$, in which case each cell domain is $P_c=\Rr$. If  $n_0 < n$, then this naturally involves other choices, namely of the cell domains. From the definition of admissible vector field, the relation $\sim_{\CC}$ must indicate the compatibility between cell domains; hence, the coarsest relation $\sim_{\CC}^0$ we can define on $\CC$ is 
$$ c \sim_{\CC}^0 d \quad \Leftrightarrow \quad P_c=P_d,$$ 
for $P_1 \times \ldots \times P_{n_0}=\Rr^n, \ P_c\neq \{0\},  \forall c \in \CC$. Let $\sim_{\CC}$ be  {a refinement of or equal to  $ \sim_{\CC}^0$ on $\CC$, that is,
$$c \sim_{\CC} d \Longrightarrow c \sim_{\CC}^0 d.$$} 
{An edge from a vertex $d$ to a vertex $c$ in an admissible graph shall represent that `$c$ depends on $d$'}. However, posing this condition may not be a simple task when we consider multiarrows. The procedure is supported by the following two mappings: 

For each $c \in \CC$,  consider the canonical submersion 
\begin{eqnarray}
\pi_c \colon \Rr^n &\longrightarrow & \prod_{d=1}^{n_0} (P_d)^{m_{cd}} \label{eqcap1-1} \\
 x & \longmapsto & (\ldots, \stackrel{\tiny{m_{cd} {\mbox{ times }}}}{\overbrace{x_d, \ldots, x_d}}, \ldots), 
 \end{eqnarray}  
 which indicates how the variable $x_d$ repeats and the number $m_{cd}$ of this repetition. Also, for the $c$-component of $f$, $f_c$,  we  consider an associated {\it generating function} \begin{eqnarray}
 \hat{f}_c \colon {P_c \times} \prod_{d=1}^{n_0} (P_d)^{m_{cd}} \rightarrow P_c\label{eqcap1-2} 
 \end{eqnarray}  
such that 
$$\dfrac{\partial \hat{f}_c}{\partial y_e} \not\equiv 0,$$ 
for any variable $y_e$ of $\hat{f}_c,$ and   
$$f_c(x)=\hat{f}_c(x_c,\pi_c(x)), \forall x \in P.$$
Thus, the number of arrows from $d$ to $c$ is $m_{cd}.$ The simplest choice is  to consider, for every $c \in \CC$, only the variables such that ${\partial {f}_c}/{\partial x_d} \not\equiv 0$ and  each $x_d$ appearing once in $\pi_c$. This yields $\hat{f}_c=f_c$, and the resulting graph is a simple graph. 

With respect to the set $\CE$ of edges, we now establish the equivalence relation $\sim_1$. In this step,  we take all edges in distinct classes, that is,  $B_1(c,c)$ is trivial and $B_1(c,d)$ is empty if $c \neq d$ and, therefore, the equivariance condition is trivially satisfied. At this point we have the relations $\sim_{\CC}$ for cells and $\sim_1$ for edges. Let $\CG_1$ denote such  network graph. \\


\noindent \textbf{Step 2.} For each $c \in \CC$, take the unique partition of {$I(c)$},
\begin{eqnarray} I(c)= \big(K^1:=\{e^1_1(c), \ldots, e^1_{s_1}(c)\}\big)\dot{\cup} \ldots\dot{\cup}\big(K^r:=\{e^r_1(c), \ldots, e^r_{s_r}(c)\}\big), \label{step2}
\end{eqnarray}
such that $\hat{f}_c$ is invariant under all permutations of $y_{e^{r'}_1}, \ldots,y_{e^{r'}_{s_{r'}}}$, for all $r'=1, \ldots, r.$ Each $K_{r'}$ is contained in a $\sim_{\CC}$-class and it is maximal with respect to these proprieties.   Consequently, $r$ is minimal. Hence, by this construction,  $f_c$ is invariant under the group ${\bf S}_{s_1} \times \ldots \times {\bf S}_{s_r}$. From this,  we define the equivalent relation $\sim_{2}$ on $\CE$,  
$$e^i_k \sim_2 e^j_l \, \, \Leftrightarrow i=j.$$  
In this way, $B(c,c)= {\bf S}_{s_1} \times \ldots \times {\bf S}_{s_r}$ and $B(c,d)$ is empty. \\

\noindent \textbf{Step 3.} Construct an input equivalence relation, for which we use again the notation $\sim_I$, and an equivalence relation $\sim_3$ on $\CE$. Here the components of the vector field are compared. For distinct  $c, d \in \CC$, $c \sim_{\CC} d$, consider the partitions as in (\ref{step2}) constructed in step 2,
\begin{eqnarray}
I(c)= \{e^1_1(c), \ldots, e^1_{s_1}(c)\}\dot{\cup} \ldots\dot{\cup}\{e^r_1(c), \ldots, e^r_{s_r}(c)\}, \nonumber \\
I(d)= \{e^1_1(d), \ldots, e^1_{q_1}(d)\}\dot{\cup} \ldots\dot{\cup}\{e^p_1(d), \ldots, e^p_{q_p}(d)\}. \nonumber
\end{eqnarray}
{If $r=p$, $s_i=q_i$, $\forall i=1, \ldots r$, and   
\begin{eqnarray}
\hat{f}_c(x_c, \overline{y_{e^1_1(c)}, \ldots, y_{e^1_{s_1}(c)}}, \ldots,\overline{y_{e^r_1(c)}, \ldots, y_{e^r_{s_r}(c)}} ) =\hat{f}_d(x_c, \overline{y_{e^1_1(c)}, \ldots, y_{e^1_{s_1}(c)}}, \ldots,\overline{y_{e^r_1(c)}, \ldots, y_{e^r_{s_r}(c)}} ), \label{step3}
\end{eqnarray}
define the equivalence relation $\sim_3$ on $\CE$, 
\begin{eqnarray} e^i_k(c) \sim_3 e^j_l(d) \, \, \Leftrightarrow i=j, \, \,  \label{step3E}
\end{eqnarray}
which gives the input equivalence $c \sim_I d$. On the other hand, if either $r \neq p$ or $s_i \neq q_i$ for some $i$, or if \eqref{step3} is not satisfied, then $c$ can not be input-equivalent to $d$.} \\

\noindent \textbf{Step 4.} Construct $\sim_4$. From the input equivalence relation obtained in step 3,  if $c \not\sim_I d,$  then $e(c) \sim_4 e'(d)$ has no effect from the point of view of the vector field, except that $c \sim_{\CC} d$. Therefore, edges can be attributed to the same $\sim_4$-class as long as the input equivalence classes are unchanged. This attribution provides further reduction on the number of classes of step 3. 

\begin{obs} \label{rmk: steps}
{\rm \noindent (a) Each $I$-class $Q$ determines uniquely the natural numbers $r=:r(Q), s_1,\ldots, s_r$ in \eqref{step3}. \\

\noindent (b) The relation \eqref{step3} may not be uniquely satisfied. This is the case if for example a generating function $\hat{f}_c$ is setwise invariant: suppose that $s_u=s_v$ and \begin{eqnarray}\hat{f}_c(\ldots, \overline{y_{e^u_1}, \ldots, y_{e^u_{s_u}}}, \ldots,\overline{y_{e^v_1}, \ldots, y_{e^v_{s_v}}} ,\ldots)  = \hat{f}_c(\ldots, \overline{y_{e^v_1}, \ldots, y_{e^v_{s_v}}}, \ldots,\overline{y_{e^u_1}, \ldots, y_{e^u_{s_u}}} ,\ldots), \label{eq: collection of K^u}
\end{eqnarray}
	for some $1\leq u,v \leq r,$ that is, the two sets  $K^u$ and $K^v$ can be interchanged in the component $f_c$. In this case, there are two possible choices in \eqref{step3} and, therefore, two choices for \eqref{step3E}. More generally, for each $I$-class $Q$, there may exist $c \in Q$ such that the collection  $\{K^1, \ldots, K^r\}$ in \eqref{step2} can be partitioned into sets 
 \begin{equation}  \label{eq: sets}
 \{K^1, \ldots, K^{u_1}\}, \{K^{u_1+1}, \ldots, K^{u_1+u_2}\},\ldots, \{K^{u_1+ \ldots+ u_{v-1}+1}, \ldots,K^{u_1+ \ldots +u_{v}}\},
 \end{equation}
 where $u_1+ \ldots +u_{v}=r$, such that $\hat{f}$ is invariant by  permutations among the sets $K^{u_1+ \ldots+ u_{t-1}+1},$  $ \ldots, K^{u_1+ \ldots+ u_t}$, for all $t=1, \ldots , v$. In particular, 
 \begin{equation} \label{eq: sizes of the sets}
 | K^{u_1+ \ldots+ u_{{t}-1}+1}|= \ldots=|K^{u_1+ \ldots+ u_{{t}}}|, \ \  t=1, \ldots, v.
 \end{equation}
 This remark also yields Proposition~\ref{propvartheta}, which gives an upper bound for the number of admissible graphs for a vector field. \\

 \noindent (c) The graphs of step 4 are optimized graphs, in the sense that they contemplate all the permutation invariances of the components of the vector field with the least number of edge types. \\ 
 }
 \end{obs}
 
 \noindent (d) As expected, the groupoid of symmetries of the graphs obtained in step 4 may not comprise all the symmetries in the components of $f$. In fact, these extra symmetries may lead to distinct ODE-classes of admissible graphs (see Theorem~\ref{mainteo}). 

\subsection{Example: a vector field on $\Rr^4$} \label{subsec:vf on R4}
Consider the vector field $f$ on $\Rr^4$ whose components are
\begin{equation} \label{eq:vc on R4}
	\begin{array}{lll}
	f_1(x_1,x_2,x_3,x_4)&=& x_1x_2+x_3x_4  \\
	f_2(x_1,x_2,x_3,x_4)&=& x_1x_2x_3x_4 \\
	f_3(x_1,x_2,x_3,x_4)&=& x_3x_4+x_1x_2  \\
	f_4(x_1,x_2,x_3,x_4)&=& x_1x_2x_3x_4. 
	\end{array}
\end{equation}	
As we shall see, this vector field admits admissible simple graphs with four, three and two cells. 
\\

\noindent \textbf{4 cells.} 
Step 1: The number of cells is the domain dimension. So $\CC = \{1,2,3,4\}$. Since for all $i,j = 1, \ldots 4$  ${\partial f_i}/{\partial x_j}$ is not identically zero, there must be an edge from any vertex to any other vertex, so $\CE = \CC \times \CC$. The network graph $\CG_1$ is then the simple complete graph with 12 possibly distinct arrows. Step 2: Starting with $f_1$, $x_1$ is the distinguished variable, so the unique permutation invariance is over the variables $x_3$ and $x_4$, so $(3,1) \sim_2 (4,1)$. For $f_2$, $x_2$ is the distinguished variable and $f_2$ is invariant under permutation over $x_1, $ $x_3$ and $x_4$, so $(1,2) \sim_2 (3,2) \sim_2 (4,2)$. Similar constructions hold for $f_3$ and $f_4$, respectively. Hence,
\begin{equation} \nonumber
  I(1) = \{ 2\} \cup \{3, 4\}, \ \ \ 
 I(2) =\{1, 3, 4\}, \ \ \
 I(3) = \{ 4\} \cup \{1, 2\}, \ \ \ 
 I(4) = \{ 1, 2, 3\},
 \end{equation}
and the vertex groups are $B(1,1)$, $B(3,3)$ isomorphic to ${\rm S}_1 \times {\rm S}_2 \simeq {\rm S}_2$ and  $B(2,2),  B(4,4)$ isomorphic to ${\rm S}_3$. The graph is $\CG_2$ given in Fig.~\ref{fig:10}, with six $\sim_2$-classes. 
Step 3: This is the correlation among components and we have that $1 \sim_2 3 \not\sim_2 2 \sim_2 4$. It then follows that the new graph $\CG_3$ is given in Fig~\ref{fig:11}, with three $\sim_{3}$-classes. Step 4: Up to  input equivalence, $\sim_3$ can be refined in two distinct ways. In fact, $1 \sim_I 3 \nsim_I 2 \sim_I 4$, so the edges between vertices 1 and 2  (or 3 and 4) can be taken in distinct edge classes (graph $\CG_4^1$ in Fig.~\ref{fig:12}) or in the same edge class (graph $\CG_4^2$ in Fig.~\ref{fig:12}). \\

\begin{figure}[h] 
	\centering 
	\includegraphics[width=3cm]{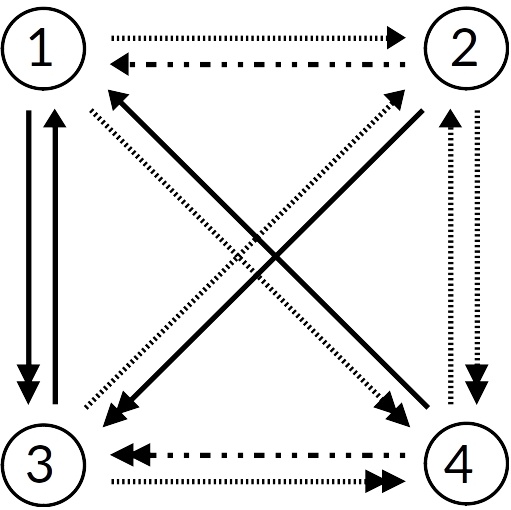} 
		\caption{\small Admissible graph $\CG_2$ for (\ref{eq:vc on R4}) with six $\sim_2$-classes.}
\label{fig:10}
\end{figure}

\begin{figure}[h] 

	\centering 
	\includegraphics[width=3cm]{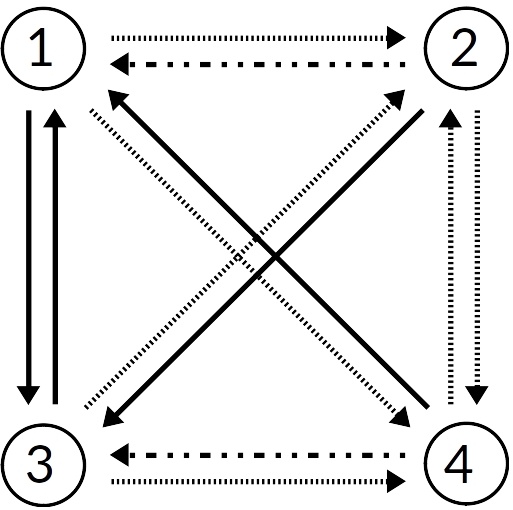} 
	\caption{\small Admissible graph $\CG_3$ for (\ref{eq:vc on R4}) with three $\sim_3$-classes.   }
	\label{fig:11}
\end{figure}

\begin{figure}[h] 
    \centering 
	\includegraphics[width=6.5cm]{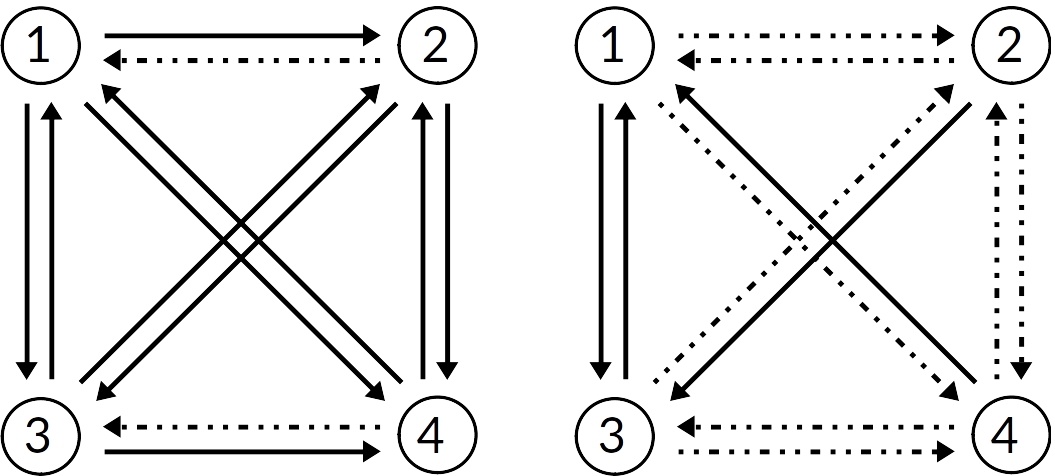}
	\caption{\small Admissible graphs $\CG_4^1$ (left) and $\CG_4^2$ (right)  for (\ref{eq:vc on R4}) with two $\sim_4$-classes.   }
	\label{fig:12}
\end{figure}

\noindent \textbf{3 cells.} Step 1: Choose the cell domains to be $P_1=\Rr^2$ with coordinates $y_1=(x_1,x_2)$, and $P_2=\Rr$, $P_3=\Rr$ with coordinates $y_2=x_3$, $y_3=x_4$, respectively. 
So $\CC = \{ 1,2,3\}$ and the $\CC$-classes are $\{1\}$ and $\{2,3\}$. Rewrite (\ref{eq:vc on R4}) as
	\begin{eqnarray}
	g_1(y_1,y_2,y_3)&=& (x_1x_2+x_3x_4,x_1x_2x_3x_4) \nonumber \\
	g_2(y_1,y_2,y_3)&=& x_1x_2+x_3x_4 \nonumber \\
	g_3(y_1,y_2,y_3)&=& x_1x_2x_3x_4. \nonumber 
	\end{eqnarray}
In step 2, we obtain the graph of Fig.~\ref{fig14}. {Regarding the vertex groups, we have $B(1,1) \simeq {\bf S}_2$. Also,  $I(2)$ is formed by two arrows of different types, so these can not be permuted and the unique bijection of $B(2,2)$ is the identity. The same goes for $I(3)$. Hence, $B(2,2)\simeq B(3,3) \simeq \{I\}$. } 
	Step 3 gives the three $I$-classes $\{1\}, \{2\}, \{3\}$. And moving to step 4, only one graph is deduced, by $(2,3) \sim_{\CE} (3,2)$; see Fig~\ref{fig15}. 
 
	\begin{figure}[H]
		\centering %
		\includegraphics[width=5cm]{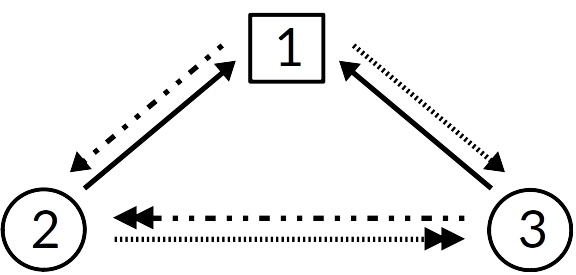} 
			\caption{\small An admissible graph $\CG_2$ for (\ref{eq:vc on R4}) seen as a network of 3 cells.}
		\label{fig14}
	\end{figure}
	
	\begin{figure}[H]
		\centering
		\includegraphics[width=5cm]{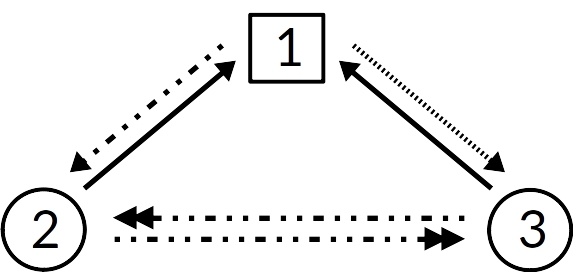} 
		\caption{\small An admissible graph $\CG_4$ for (\ref{eq:vc on R4})  seen as a network of {3} cells.}
		\label{fig15}
		\end{figure}

\noindent 	\textbf{2 cells}: In this case, we choose the cell phase spaces $P_1=\Rr
	^2$ with coordinate $y_1=(x_1,x_2)$ and $P_2=\Rr
	^2$ with coordinate $y_2=(x_3,x_4)$ and rewrite (\ref{eq:vc on R4}) as 
	
	\begin{eqnarray}h_1(y_1,y_2)=(f_1,f_2)=(x_1x_2+x_3x_4,x_1x_2x_3x_4) \nonumber \\
	h_2(y_1,y_2)=(f_3,f_4)=(x_1x_2+x_3x_4,x_1x_2x_3x_4). \nonumber
	\end{eqnarray}
	Clearly the optimized admissible graph for this network is as in Fig.~\ref{fig13}.
	
	\begin{figure}[h]
		\centering 
		\includegraphics[width=4cm]{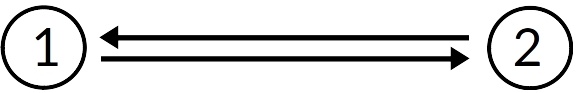} 
		\caption{\small An admissible graph $\CG_4$ for (\ref{eq:vc on R4}) seen as a network of 2 cells. }
		\label{fig13}
	\end{figure}

\subsection{Simple graph versus multigraph} \label{subseq: simple graph versus multigraph}
In this subsection we finish the discussion started in  Section~\ref{sec:introduction} for the example (\ref{eq:3-cell}): we construct the possible network graphs using the procedure of Subsection~\ref{subseq:method} for the vector field on $\Rr^3$ with components 
\begin{eqnarray}
f_1(x_1,x_2,x_3)&=&x_1+x_1^3 \nonumber \\
f_2(x_1,x_2,x_3)&=&x_2+x_2^2x_3 \label{eqca1-3} \\
f_3(x_1,x_2,x_3)&=&x_3+x_1x_2x_3. \nonumber 
\end{eqnarray}

\noindent  \textbf{Simple graph.} Step 1: we have that $\CC = \{ 1, 2, 3 \}$ with valencies $0, 1, 2$ and it makes it consistent to take $1\not\sim_{I} 2 \not\sim_{I} 3 \not\sim_{I} 1$. Step 2: this leads easily to 
$(1,3) \sim_{2} (2,3)$. The input and the edges equivalences are unchanged in step 3. Step 4: this leads easily to $(1,3) \sim_{4} (2,3) \sim_{4}(3,2)$, producing an optimized simple graph; see Fig.\ref{graph1} (left). \\

\noindent \textbf{Multigraph.} The function $\hat{f}_c : {\Rr^4} \to \Rr,$ 
$$\hat{f}_c(x,\overline{y_1,y_2,y_3})= x+ y_1y_2y_3$$
is a generating function of all components $f_c$, $c = 1, 2, 3$, in (\ref{eqca1-3}),  
\begin{eqnarray}
f_1(x_1,x_2,x_3)&=&\hat{f}(x_1,\overline{x_1,x_1,x_1}) \nonumber \\
f_2(x_1,x_2,x_3)&=&\hat{f}(x_2,\overline{x_2,x_2,x_3})  \\
f_3(x_1,x_2,x_3)&=&\hat{f}(x_3,\overline{x_1,x_2,x_3}). \nonumber 
\end{eqnarray} 
Now it is straightforward to see that the resulting admissible graph is the regular multigraph of Fig.\ref{graph1} (right).

\section{Relation among admissible graphs} \label{sec:main}
The central questions raised about the realization of admissible graphs of Section~\ref{sec:realization} regard the number of possible graphs as well as the relation among them, with respect to isomorphism and ODE-equivalence. Here we present the results providing the answers. \\

We start the section with an example of two non-isomorphic but ODE-equivalent admissible graphs. These are the 4-cell network graphs of Subsection~\ref{subsec:vf on R4} obtained in step 4 of the procedure. Let $\CG_4^1$ be the graph  of Fig.~\ref{fig:12} on the left. Its adjacency matrices are
	\begin{eqnarray}A_{\CG_1}^{\rightarrow}=\left[
	\begin{array}{cccc}
	0 & 0 & 1 & 1\\
	1 & 0 & 1 & 1\\
	1 & 1 & 0 & 0\\
	1 & 1 & 1 & 0 
	\end{array} \right], \quad A_{\CG_1}^{\dashrightarrow}=\left[
	\begin{array}{cccc}
	0 & 1 & 0 & 0\\
	0 & 0 & 0 & 0\\
	0 & 0 & 0 & 1\\
	0 & 0 & 0 & 0 
	\end{array} \right]. \nonumber 
	\end{eqnarray} 
Hence,
	\begin{eqnarray}
	\CL(\CG_4^1;\Rr^n)  = \left\{ \left[
	\begin{array}{cccc}
	t & t^{\dashrightarrow} & t^{\rightarrow} & t^{\rightarrow}\\
	s^{\rightarrow} & s & s^{\rightarrow} & s^{\rightarrow}\\
	t^{\rightarrow} & t^{\rightarrow} & t & t^{\dashrightarrow}\\
	s^{\rightarrow} & s^{\rightarrow} & s^{\rightarrow} & s 
	\end{array} \right] \ : \  t,s, t^{\rightarrow},s^{\rightarrow}, t^{\dashrightarrow} \in \Rr \right\}. \nonumber  
	\end{eqnarray}
Now let $\CG_4^2$ be the other graph in Fig.~\ref{fig:12} (right). Its adjacency matrices are
	\begin{eqnarray}A_{\CG_2}^{\rightarrow}=\left[
	\begin{array}{cccc}
	0 & 0 & 1 & 1\\
	0 & 0 & 0 & 0\\
	1 & 1 & 0 & 0\\
	0 & 0 & 0 & 0 
	\end{array} \right], \quad A_{\CG_2}^{\dashrightarrow}=\left[
	\begin{array}{cccc}
	0 & 1 & 0 & 0\\
	1 & 0 & 1 & 1\\
	0 & 0 & 0 & 1\\
	1 & 1 & 1 & 0 
	\end{array} \right]. \nonumber 
	\end{eqnarray} 
Hence,
	\begin{eqnarray}
	\CL(\CG_4^2;\Rr^n)  = \left\{ \left[
	\begin{array}{cccc}
	t & t^{\dashrightarrow} & t^{\rightarrow} & t^{\rightarrow}\\
	s^{\dashrightarrow} & s & s^{\dashrightarrow} & s^{\dashrightarrow}\\
	t^{\rightarrow} & t^{\rightarrow} & t & t^{\dashrightarrow}\\
	s^{\dashrightarrow} & s^{\dashrightarrow} & s^{\dashrightarrow} & s 
	\end{array} \right] \ : \  t,s, t^{\rightarrow}, t^{\dashrightarrow},s^{\dashrightarrow} \in \Rr \right\}. \nonumber  
	\end{eqnarray}
Therefore,  
$$\CL(\CG_4^1;\Rr^n)=\CL(\CG_4^2;\Rr^n),$$
so $\CG_4^1$ and $\CG_4^2$ are ODE-equivalent. But they  are clearly non isomorphic. \\

As we have already seen, step 4 refines the relations of step 3, when the quantity of admissible graphs is attained: 

\begin{prop} \label{propvartheta} The number of possible distinct equivalence relations $\sim_{3}$ (step 3) is given by
	\begin{eqnarray}
	\vartheta(f)=\prod_{Q \in \CC/_{\sim_I}}\left(\prod_{t=1}^v (u_t!)^{(| Q|-1)}\right).  \label{obsstep3}
	\end{eqnarray} 
	\end{prop}

\begin{proof} {
We use the notation of Remark~\ref{rmk: steps} (b). For each $I$-class $Q$, we have $u_1, \ldots, u_v$ uniquely determined. Based on the possible partitions of (\ref{eq: sets}) which  satisfy (\ref{eq: sizes of the sets}), it follows that,  for each $c \in Q$, we have $v$ collections of edge sets \eqref{eq: collection of K^u} such that the sets in each collection can be permuted in $\hat{f}_c$. Hence, once we choose $c=c_0 \in Q$ as a reference,  for each $ d \in Q-\{c_0\}$ the equation \eqref{step3} has $$\prod_{t=1}^v u_t!$$
ways to be satisfied. }
\end{proof}

From a fixed choice of the associated generating functions, the next three results specify how the resulting admissible graphs are related.

\begin{teo} \label{mainteo}
For a given $C^1$ vector field, the realization procedure of Subsection~\ref{subseq:method} yields admissible graphs with the following relations:
	\begin{itemize} 
		\item [(1)] The network graph $\CG_1=(\CC,\CE, \sim_{\CC}, \sim_1)$ of step 1 and the network graphs $\CG_2=(\CC,\CE, \sim_{\CC}, \sim_2)$ of step 2 are unique up to  an isomorphism;
		
		\item [(2)] Each $\CG_4=(\CC,\CE, \sim_{\CC}, \sim_4)$ of step 4 may be non isomorphic to another network graph obtained in this step, but it is ODE-equivalent to some $\CG_3$ of step 3;
		\item [(3)] The following inclusions hold: $$\CL(\CG_4;\Rr^{n_0}) = \CL(\CG_3;\Rr^{n_0}) \subseteq \CL(\CG_ 2;\Rr^{n_0}) \subseteq \CL(\CG_1;\Rr^{n_0}).$$
	\end{itemize} 
\end{teo}

\begin{proof}

We prove by remaking the four steps of the procedure from the adjacency matrix point of view. For simplicity, the proof is carried out for the construction of simple graphs. For multigraphs the proof is completely analogous adapting the notation of an edge which  is not identified with a pair of vertices and so adjacency matrices have integer entries instead of only 0 and 1.

In step 1, for a choice of $\CC$, there is a one-to-one correspondence between the edges in $\CE$ and adjacency matrices: each adjacency matrix $A_1^{i,j}$ corresponds to  $(i,j) \in \CE$, with the $ji$-entry equal to 1 and the others equal to zero. Such matrices are unique up to vertex labeling. 
	
In step 2, for each $c \in \CC$ consider the partition \eqref{step2} of $I(c)$. The adjacency matrices are, for each $r'=1, \ldots, r$, 
$$A_2^{c,r'} = \sum_{d \in K_{r'}} A_1^{d,c},$$ 
so that the graph is unique up to an isomorphism. 
	
In step 3, we construct the $I$-classes of cells. For each $I$-class $Q$, the adjacency matrices are 
$$A_3^{Q,r'} = \sum_{c \in Q} A_2^{c,r'} \quad r'=1, \ldots, r.$$
As already registered in Remark~\ref{rmk: steps} (b), such matrices depend on the ordination of variables in each component; in addition,  {$\vartheta(f)$} is the number of all possible ordinations, by Proposition~\ref{propvartheta}; so distinct choices may lead to isomorphic graphs. 

At this point, it is already straightforward to see that the two inclusions in item (3) hold. So we finally prove the statement in (2).

Without loss of generality, assume that the graph $\CG_3$ has  two $I$-classes $R=\{1, \ldots,m\}, S=\{m+1, \ldots,n_0\}$. Since  $R \cap S = \emptyset,$ then the null rows of $A_3^{R,r'}$ are non null rows in $A_3^{S,s'}$, and vice versa.  The adjacency matrices of a graph $\CG_4$ in step 4 {are} as follows: 
	\begin{eqnarray} 
	&A_3^{R,r''}&, \quad r''=1, \ldots,r  \mbox{ and } r'' \neq r', \nonumber \\ 
	&A_3^{S,s''}&, \quad s''=1, \ldots,s  \mbox{ and } s'' \neq s', \nonumber \\
	&B_4=A_3^{R,r'}+A_3^{S,s'}&. \nonumber 
	\end{eqnarray}
Thus, $M \in \CL(\CG_4,\Rr^{n_0})$ if, and only if,	
\begin{eqnarray}
	M&=&  diag(\stackrel{\tiny{m {\mbox{ times }}}}{\overbrace{t, \ldots,t}},\stackrel{\tiny{n_0-m {\mbox{ times }}}}{\overbrace{\bar{t}, \ldots, \bar{t}}})+\sum_{r''\neq r'} diag(t^{r''}, \ldots, t^{r''},0, \ldots, 0)A_3^{R,r''} + \nonumber \\
	& & +\sum_{s''\neq s'} diag(0, \ldots, 0,t^{s''}, \ldots, t^{s''})A_3^{S,s''} \nonumber \\
	& & + diag(t^{r'}, \ldots, t^{r'},t^{s'}, \ldots, t^{s'})\big( A_3^{R,r''}+A_3^{S,s''} \big) \nonumber \\
	&=&  diag(t, \ldots,t,\bar{t}, \ldots, \bar{t})+\sum_{r''} diag(t^{r''}, \ldots, t^{r''},0, \ldots, 0)A_3^{R,r''} + \nonumber \\
	& & +\sum_{s''} diag(0, \ldots, 0,t^{s''}, \ldots, t^{s''})A_3^{S,s''},
\end{eqnarray}
which is a general element of 
$\CL(\CG_3; \Rr^{n_0}).$ Therefore, $\CL(\CG_4; \Rr^{n_0}) = \CL(\CG_3; \Rr^{n_0}).$ This equality holds for any choice of the pair of indices $r',s'$ as long as the $I$-classes are kept unchanged, so this concludes the proof.  
\end{proof}

In the remaining of this section we discuss how the $\vartheta(f)$ graphs of step 3 are related according to the ODE-equivalence. We assume that the vector field is smooth and that the cells of the resulting network graphs with, say, $n$ cells, have the same dimension. 

Without loss of generality, we assume that each cell phase space is one-dimensional. To ease the exposition, we present the details for the simplest but sufficiently general case: For a smooth vector field  $f \colon \Rr^n \rightarrow \Rr^n$, suppose that step 3 results in two input classes $\{1, \ldots,n_0\}$ and $\{n_0+1, \ldots,n\}$ with associated {generating functions} $g$ and $h$,  
\begin{equation} \label{eq: generating functions}
\begin{array}{rl}
g(y,\overline{y_1,\ldots,y_k},\overline{y_{k+1},\ldots,y_{2k}},\overline{y_{2k+1},\ldots,y_l}) &  =  g(y,\overline{y_{k+1},\ldots,y_{2k}},\overline{y_1,\ldots,y_k},\overline{y_{2k+1},\ldots,y_l}) \\
h(y,\overline{y_1,\ldots,y_{k'}},\overline{y_{k'+1},\ldots,y_{2k'}}) & = h(y,\overline{y_{k'+1},\ldots,y_{2k'}},\overline{y_1,\ldots,y_{k'}}). 
\end{array}
  \end{equation}  
  We have $\vartheta(f)=2^{n_0-1}2^{n-n_0-1}=2^{n-2}$ graphs, and let $\CG$ be one of the graphs. Then $\CG$ has five adjacency matrices $A_{\CG}^i$, $i = 1, \ldots, 5$, three of which (say for $i = 1,2, 3)$  having the last $n - n_0$ null rows  and two ($i=4,5$) having the first $n_0$ null rows. \\

{We now give a technical definition. 
  \begin{defin} \label{defin barG}
With the notation of Remark \ref{rmk: steps}(b), let $\CG$ be  a network graph determined by the partition (\ref{eq: sets}). The network graph $\bar{\CG}$ is given by defining  its $\CE$-classes as the unions
$$K^1\cup \ldots\cup K^{u_1},  \ K^{u_1+1}\cup \ldots\cup K^{u_1+u_2},\ldots, \ K^{u_1+ \ldots+ u_{v-1}+1}\cup \ldots\cup K^{u_1+ \ldots +u_{v}},$$
for each $I$-class $Q$ and each $c \in Q$.
\end{defin}}

For the case (\ref{eq: generating functions}), $\bar{\CG}$ is the graph whose adjacency matrices are 
\[A_{\CG}^1+ A_{\CG}^2, \ A_{\CG}^3, \ A_{\CG}^4+ A_{\CG}^5.\]

\begin{obs}
{\rm The graph $\bar{\CG}$  does not depend on the particular choice of  $\CG$, but only on the associated generating functions.} 
\end{obs}

Let $\gamma \colon \CC \rightarrow \CC$ be a bijection that preserves the input classes of $\CG$. The graph $\gamma \CG$, whose adjacency matrices are $\gamma A_{\CG}^i\gamma^{-1},$ for $i = 1, \ldots, 5$, is ODE-equivalent  to $\CG$. However, it is not necessarily true that $\gamma \CG$ is also an admissible graph for this vector field. So under what conditions over $\gamma$ is this an admissible graph?  The following proposition gives a necessary condition:

\begin{prop} \label{prop: G bar}
For generic generating functions  and for $\CG$ an admissible  graph, if $\gamma \CG$ is an admissible graph, then there exists $\gamma'=(\gamma_{\CC},\gamma_{\CE}) \in Aut(\bar{\CG})$ such that $\gamma=\gamma_{\CC}$.
\end{prop}

\begin{proof}
To ease exposition, we present the proof for the case in (\ref{eq: generating functions}). For any $a \in \Rr$, we take  $\nu=(a,\ldots,a) \in \Rr^n$. Let us also denote by $\nu$ the vectors in $\Rr^{l+1}$ and $\Rr^{2k'+1}$ with all entries equal to $a$. By the permutation symmetries of $g,h$, we have that 
$$\frac{\partial g}{\partial y_1}(\nu)= \frac{\partial g}{\partial y_{k+1}}(\nu), \quad \frac{\partial h}{\partial y_1}(\nu)= \frac{\partial h}{\partial y_{k'+1}}(\nu).$$
Since $f$ is $\CG$-admissible, these imply that 
\begin{eqnarray} Df(\nu)= \alpha + \frac{\partial g}{\partial y_1}(\nu)(A_{\CG}^1+A_{\CG}^2)+\frac{\partial g}{\partial y_{2k+1}}(\nu)A_{\CG}^3+  \frac{\partial h}{\partial y_1}(\nu)(A_{\CG}^4+A_{\CG}^5), \nonumber
\end{eqnarray}
where $\alpha$ is a diagonal matrix whose first $n_0$ entries are equal to $\dfrac{\partial g}{\partial y}(\nu)$ and the last $n-n_0$ entries are equal to $\dfrac{\partial h}{\partial y}(\nu)$. 
Similarly, as $f$ is $\gamma \CG$-admissible, then 
\begin{eqnarray} Df(\nu) = \alpha + \frac{\partial g}{\partial y_1}(\nu)\gamma(A_{\CG}^1+A_{\CG}^2)\gamma^{-1} +\frac{\partial g}{\partial y_{2k+1}}(\nu)\gamma A_{\CG}^3\gamma^{-1} +  \frac{\partial h}{\partial y_1}(\nu)\gamma(A_{\CG}^4+A_{\CG}^5)\gamma^{-1}. \nonumber
 \end{eqnarray}
Comparing the null rows, the following equalities hold:  
\begin{eqnarray}
\frac{\partial g}{\partial y_1}(\nu)(A_{\CG}^1+A_{\CG}^2)+\frac{\partial g}{\partial y_{2k+1}}(\nu)A_{\CG}^3 &=& \frac{\partial g}{\partial y_1}(\nu)\gamma(A_{\CG}^1+A_{\CG}^2)\gamma^{-1} +\frac{\partial g}{\partial y_{2k+1}}(\nu)\gamma A_{\CG}^3\gamma^{-1} \nonumber  \\  
\frac{\partial h}{\partial y_1}(\nu)(A_{\CG}^4+A_{\CG}^5) &=& \frac{\partial h}{\partial y_1}(\nu)\gamma(A_{\CG}^4+A_{\CG}^5)\gamma^{-1} \nonumber 
\end{eqnarray}
If there exist  $a_1,a_2,a_3$ such that 
\begin{eqnarray} det \left[ \begin{array}{cc} \frac{\partial g}{\partial y_1}(\nu_1) & \frac{\partial g}{\partial y_{2k+1}}(\nu_1) \\ \frac{\partial g}{\partial y_1}(\nu_2) & \frac{\partial g}{\partial y_{2k+1}}(\nu_2)    \end{array} \right] \neq 0,  \quad  \frac{\partial h}{\partial y_1}(\nu_3) \neq 0, \label{eq: generic condition f} \end{eqnarray}
(this is the generic condition on $g, h$), then  
\begin{eqnarray} A_{\CG}^1+A_{\CG}^2=\gamma(A_{\CG}^1+A_{\CG}^2)\gamma^{-1}, \quad A_{\CG}^3=\gamma A_{\CG}^3\gamma^{-1}, \quad A_{\CG}^4+A_{\CG}^5=\gamma(A_{\CG}^4+A_{\CG}^5)\gamma^{-1}.\nonumber \end{eqnarray}
But this means that $\gamma$ is a bijection on the set of cells of $\CG$ by an automorphism of $\bar{\CG}$.
\end{proof}

\begin{teo} \label{teomain2} Let $\gamma'=(\gamma_{\CC},\gamma_{\CE}) \in Aut(\bar{\CG})$. \begin{itemize}
\item [(1)]  The vector field $f$ is $\CG$-admissible and $\gamma_{\CC}\CG$-admissible if, and only if,  the bijection $\gamma_{\CE}$ induces a bijection on the input sets (step 2) preserving them as partitions, that is, 
\begin{eqnarray} \label{eq:partitions}
\gamma_{\CE}(I(c))=I(\gamma_{\CC}(c)), \mbox{ as partitions. } \label{eq002} 
\end{eqnarray}
\item [(2)] The condition  (\ref{eq:partitions}) on $\gamma'$ implies that $f \gamma_{\CC} = \gamma_{\CC} f$. The converse holds for the simple graph case. 
\end{itemize}
\end{teo}
\begin{proof}
(1) If \eqref{eq002} holds, then it is direct from Definition~\ref{definAVF} that $f$ is $\gamma_{\CC}\CG$-admissible, since $\gamma_{\CC}$ preserves input sets. If \eqref{eq002} does not hold, then there exists $c \in \CC$ such that $\gamma_{\CE}(I(c))\neq I(\gamma_{\CC}(c))$;  by uniqueness of the partition of $I(\gamma_{\CC}(c)),$ it follows that $f$ is not $\gamma_{\CC}\CG$-admissible.  

(2) For $c \in \{1, \ldots, n_0\}$, we have that
$$
(f \gamma_{\CC})_c(x_1, \ldots,x_n)= f_c(x_{\gamma_{\CC}^{-1}(1)}, \ldots, x_{\gamma_{\CC}^{-1}(n)})=g(x_{\gamma_{\CC}^{-1}(c)}, x_{\gamma_{\CE}^{-1}(I(c))}) $$
and
$$ (\gamma_{\CC}f )_c(x_1, \ldots,x_n)=f_{\gamma_{\CC}^{-1}(c)}(x_1, \ldots,x_n)=g(x_{\gamma_{\CC}^{-1}(c)}, x_{I(\gamma_{\CC}^{-1}(c))}), $$
so the equality $f \gamma_{\CC} = \gamma_{\CC} f$ follows from \eqref{eq002}. For the converse, suppose that  $f \gamma_{\CC} = \gamma_{\CC} f$, so 
$$g(x_{\gamma_{\CC}^{-1}(c)}, x_{\gamma_{\CE}^{-1}(I(c))})= g(x_{\gamma_{\CC}^{-1}(c)}, x_{I(\gamma_{\CC}^{-1}(c))}).$$ 
But \eqref{eq002} to fail would contradict the maximality of the parts of $I(\gamma_{\CC}^{-1}(c))$ for simple graphs.   
\end{proof} 
 
\begin{coro} \label{cor: Iso}
    Let $\CG$ be an admissible simple graph of step 3 for $f$. Any admissible graph for $f$ and nonisomorphic to $\CG$ is ODE-equivalent to $\CG$ if, and only if, it is of the form $\sigma \CG$, where $\sigma$  belongs to  
$$\Sigma (\CG) \ = \ \{ \gamma \in Aut(\bar{\CG}) \tq \gamma f = f \gamma \} / Iso(\CG).$$

\end{coro}
\begin{proof}
This is direct from Theorem~\ref{teomain2} (2).
\end{proof}
The example of the next subsection illustrates that  distinct choices of step 3 can lead to isomorphic graphs or also to non ODE-equivalent graphs. 

\subsection{Example: a vector field on $\Rr^6$} \label{sec: vf on R6}

The aim of this subsection is to show with an example  that a vector field may admit invariant polydiagonal subspaces  that are not realized as a synchrony pattern of an admissible graph of this vector field. Nevertheless, for this particular example of six cells we shall verify that the polydiagonal invariant subspaces are generically realized as synchronies of the network graph $\bar{\CG}$ presented above. \\

We construct admissible simple graphs for vector fields $f: {\Rr}^6 \to {\Rr}^6$  that govern systems of the form
\begin{equation} \label{eq:vf on R6}
\begin{array}{ll}
\dot{x_1} = & g(x_1, x_5, x_6, x_2, x_3) \\
\dot{x_2} = & g(x_2, x_6, x_1, x_3, x_4) \\
\dot{x_3} = & g(x_3, x_1, x_2, x_4, x_5) \\
\dot{x_4} = & g(x_4, x_2, x_3, x_5, x_6) \\
\dot{x_5} = & g(x_5, x_3, x_4, x_6, x_1) \\
\dot{x_6} = & g(x_6, x_4, x_5, x_1, x_2), 
\end{array} 
\end{equation}
for $g : \Rr^5 \to \Rr$ such that
\begin{eqnarray} g(y, \overline{y_1, y_2}, \overline{y_3, y_4}) =
g(y, \overline{y_3, y_4}, \overline{y_1, y_2}). \label{eq-9.4}
\end{eqnarray} 
Notice that $\bar{\CG}$ in this case is the graph $G_6$ of six cells with nearest and next nearest neighbor coupling. {The synchrony subspaces of $G_6$ are given in Table~1.}

Clearly, an optimized admissible simple  graph for $f$ with six cells is homogeneous with two types of edges. From \eqref{eq-9.4}, there are $\vartheta(f)=2^{6-1}=32$ ways to define the edge classes. By investigation, it is easy to see that, up to isomorphism, there are eight types of admissible simple graphs. By Proposition~\ref{prop:homog simple graphs}, these are all non ODE-equivalent. In Table~2 we present the eight types of admissible graphs and, for each type representative $\CG$ listed in the first column, we give the possible synchrony subspaces. In addition, in the last column we present the number $| \Sigma(\CG)|$ of the ODE-equivalence class (see Corollary~\ref{cor: Iso}). Notice that these numbers sum up to give $\vartheta(f)$.

We finally discuss the data presented in Table~2 with respect to the invariant polydiagonal subspaces under $f$. On one hand,  
each number in the second column corresponds to a subspace that is obviously invariant under $f$, by the admissibility of the graphs. On the other hand, let $\Theta$ be a polydiagonal subspace which is invariant under $f$,
\begin{equation} \label{eq: Theta}
f(\Theta) \subseteq \Theta.
\end{equation}
{We proof that generically this} is a synchrony subspace for some graph in Table~2. Let $A^1$ and $A^2$ be the adjacency matrices of $\CG$. Since  $\Theta$ is polydiagonal, it contains the diagonal. So let 
$\nu \in \Theta$ be in the diagonal. 
By \eqref{eq-9.4}, we have that $\dfrac{\partial g}{\partial y_1}(\nu)=\dfrac{\partial g}{\partial y_3}(\nu)$, and then
\begin{eqnarray} {Df}(\nu)(\Theta)&=&\frac{\partial g}{\partial y}(\nu) \Theta+\frac{\partial g}{\partial y_1}(\nu)A^1(\Theta)+\frac{\partial g}{\partial y_3}(\nu)A^2(\Theta) \nonumber \\ &=&\frac{\partial g}{\partial y}(\nu) \Theta+\frac{\partial g}{\partial y_1}(\nu)(A^1+A^2)(\Theta). \nonumber 
\end{eqnarray}
If $\nu$ satisfies the  generic  condition $\dfrac{\partial g}{\partial y_1}(\nu) \neq 0$, then it follows directly from 
(\ref{eq: Theta}) that the inclusion $(A^1+A^2)(\Theta) \subseteq \Theta$ also holds. We now notice that $A^1+A^2$ is the adjacency matrix of the network graph $G_6$. Therefore, 
$\Theta$ is generically {induced by} a synchrony pattern of $G_6$. {Now, $G_6$ is not an admissible graph for \eqref{eq:vf on R6}, but each of its synchrony patterns (Table~1) falls into the second column of  Table~2 for some (maybe more than one) graph up to the action of ${\rm Aut}(G_6)$.} \\

{For the network of  the next section, the admissible graph is $G_6$ and we also discuss about invariant subspaces. We use a similar approach as above, but  with an extra linear algebra property provided for that particular case.} 

\section{Coupled network of van der Pol identical oscillators} \label{sec:van der Pol}
In this section we present the possible synchronous configurations in a specific network of identical oscillators. In particular, we find the hybrid states of chimera, verifying that spatially separated domains of synchronized and desynchronized behavior can arise  in networks of identical units with symmetric coupling topologies.  

We consider a network of six second-order systems of van der Pol type identical oscillators which are coupled in a non-local fashion with additional intensity-dependent frequency (\cite{CHANDRASEKAR}):
\begin{equation} 
\begin{array}{lll} \label{eq:2nd order van der Pol}
\ddot{x}_i &=& b(1-x^2_i)\dot{x}_i-(\omega_0^2+\alpha_1x^2_i+\alpha_2x^4_i)x_i  +  \epsilon(\dfrac{1}{4}\displaystyle{(\sum_{j=i-2}^{j=i+2}\dot{x}_j)-\dot{x}_i)}+  \eta(\dfrac{1}{4}\displaystyle{(\sum_{j=i-2}^{j=i+2}{x}_j)-{x}_i}{)},
\end{array}
\end{equation}
for $i=1, \ldots, 6$, where $\alpha_1$ and $\alpha_2$ are the so-called intensity parameters, $\epsilon$ and $\eta$ are the coupling strengths.  

A network of several systems of van der Pol type identical oscillators  was numerically studied in \cite{CHANDRASEKAR} for 500 oscillators. That work shows that from a random initial condition the system can  evolve to a situation where part of the oscillators are in synchrony and  the other part are in total incoherence of phases.

In 1665, Christiaan Huygens observed that two  pendulum clocks always synchronized after a certain time. The justification was that, for been stuck in the same wood, part of the momentum of one pendulum traveled as vibrations in the wood to the other pendulum. Since then, it was thought that a set of coupled oscillators either remained in total disorder or, after a certain time, got in synchrony. However, as first observed by Kuramoto and Battogtokh \cite{KB}, for the coupling of identical oscillators there can be an intermediate state between total synchrony (coherence) and total desynchrony (incoherence). In this regime,  part of the oscillators are synchronized and the other part are in total incoherence of phases. This phenomenon, known as `chimera state' after Abrams and Strogatz \cite{AS}, has since attracted attention and enormous interest in many fields of applications. Experiments and numerical simulations have given evidences that this behaviour is not only possible but also expected to be stable as the system evolves.

We investigate synchronous states of (\ref{eq:2nd order van der Pol}), namely
\begin{equation} \label{eq:xi=xj}
x_i(t)=x_j(t), \forall t \in \Rr,
\end{equation}
for $i,j$ in some subset of $\{1, \ldots, 6\}$. 

Notice that, in particular,  $\dot{x}_i(t)=\dot{x}_j(t)$. Hence, if we rewrite  (\ref{eq:2nd order van der Pol}) as the coupled Hamiltonian system 
\begin{equation} \label{eq:vanderpol1}
\begin{array}{lll}
\dot{x}_i &=&  y_i \\ 
\dot{y}_i &=& b(1-x^2_i)y_i-(\omega_0^2+\alpha_1x^2_i+\alpha_2x^4_i)x_i \\
& & + \ \epsilon (\dfrac{1}{4}(\displaystyle{\sum_{j=i-2}^{j=i+2}y_j)-y_i)}+ \eta(\dfrac{1}{4}(\displaystyle{\sum_{j=i-2}^{j=i+2}{x}_j)-{x}_i)}, 
\end{array}
\end{equation}
for $i=1, \ldots,6,$ then the synchronies of one are in one-to-one correspondence with the synchronies of the other. 

This is a vector field admissible for the regular network graph $G_6$ with six cells on $({\mathbb R}^2)^6$ with nearest and next nearest neighbor identical coupling. Including the trivial totally synchronous pattern, there are nine distinct patterns of synchrony, which have been computed from our code {(see Remark \ref{obs code})}. In Table~1 we present all the possible eight nontrivial synchronies. In this particular case, each synchrony subspace is the fixed-point subspace of a subgroup of the automorphism group Aut($G_6$) of the network graph, which is the octahedral
group  ${\bf O} \simeq \bigl<{\bf D}_6, (14), (25)\bigr>$  (see \cite[Lemma 2.1]{GNS}, where the authors present the automorphism groups Aut($G_N$),  $N \geq 5$). See also our comment in the end of Subsection~\ref{subsec:ber}. 

Regarding the data of Table~1,  contrarily to what is usual in the literature, we give all the possible algebraic expressions (second column) of the corresponding {diagram} (first column), which are obviously fixed-point subspaces of conjugate isotropies. Doing so,  we link the data of Tables~1 and 2: expressions (1) to (3) of $\sharp 1$ in Table~1  appear in Table~2 $\sharp 6$  but  not in $\sharp 2$; similarly, the expressions (19) to (22) of $\sharp 6$ in Table~1  appear in $\sharp 6$  but  only (21) appear in $\sharp 6$ of Table~2, and so on.  \\

Chimera states can now be selected directy from Table~1, namely configurations $\sharp 1$ to $\sharp 4$. For example, in case $\sharp 2$ the cells 3 and 6 are two isolated desynchronized cells. From the point of view of the applications,  the detection of the possible robust attracting chimeras from this list is a relevant issue,  for these correspond to the numerically observed phenomena of the literature mentioned in our introductory section. At the present, this investigation has been carried out. \\

By the condition in \eqref{eq: condition to balanced}, all synchrony subspaces are polydiagonal invariant  subspaces  under the vector field $f$ defined by  \eqref{eq:vanderpol1}.
We finish by investigating the converse. Let $\Theta$ be a polydiagonal subspace which is invariant under the vector field $f$ given in \eqref{eq:vanderpol1}.  

In particular, it is invariant under the linearization at the origin,
\begin{eqnarray} 
{Df}(0)=I_6 \otimes \left[ \begin{array}{cc} 0 & 1 \\ -\omega_0^2 -\dfrac{3\eta}{4} & b  -\dfrac{3\epsilon}{4}\end{array} \right] +A_{G_6} \otimes \left[ \begin{array}{cc} 0 & 0 \\ \dfrac{\eta}{4} & \dfrac{\epsilon}{4}\end{array} \right], \nonumber 
\end{eqnarray} which implies that
$$A_{G_6} \otimes \left[ \begin{array}{cc} 0 & 0 \\ \dfrac{\eta}{4} & \dfrac{\epsilon}{4}\end{array} \right][\Theta] \subset \Theta.$$
Now, since any vector $v \in \Theta$ is of the form $v=(x_1,y_1, \ldots,x_6,y_6)$ with  $x_i=x_j$ and $y_i=y_j$ for $i,j$ in some subset of $\{1, \ldots, 6\}$, and that 
\begin{eqnarray}
A_{G_6} \otimes \left[ \begin{array}{cc} 0 & 0 \\ \dfrac{\eta}{4} & \dfrac{\epsilon}{4}\end{array} \right][v] =  \frac{1}{4}A_{G_6} \otimes I_2[(0,\eta x_1+ \epsilon y_1, \ldots,0,\eta x_6+ \epsilon y_6)] \in \Theta, 
\end{eqnarray}
it follows that $\Theta'$ is an $A_{G_6}$-invariant subspace of $\Rr^6$, in the $x_i$'s variables,  defined by the same equalities as for $\Theta$.  \\

\noindent {\bf Acknowledgments.}  TAA acknowledges financial support by FAPESP grant number 2019/2130-0. MM acknowledges financial support by FAPESP grant 2019/21181-0.

\begin{table}[H]
\centering { \small{ 
\begin{tabular}{ |m{0.3cm}|m{3.5cm}|m{6cm}|m{5cm}|  } 
 \hline 
 
 \quad $\sharp$ & \quad \quad  \quad {Diagram} & \quad  \quad Algebraic expression & Symmetry of the pattern under the standard representation of the octahedral group \\
 \hline
 $1$ & \includegraphics[scale=0.25]{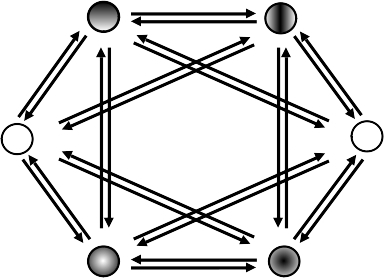}   & $(1) \; \; \{x_1=x_4\} $  \newline $(2) \; \;\{x_2=x_5\}$ \newline    $(3) \; \;\{x_3=x_6\}$     & Reflection w.r.t. the plane  through the other four cells. \\ \hline 
  $2$ & \includegraphics[scale=0.25]{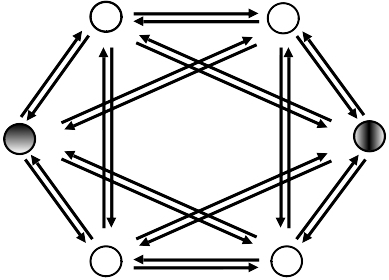}   & $(4) \; \;\{x_1=x_2=x_4=x_5\} $  \newline $(5) \; \;\{x_1=x_3=x_4=x_6\}$   \newline $(6) \; \;\{x_2=x_3=x_5=x_6\}$  & Rotation of {$\pi /2 $} w.r.t. an axis  through opposite cells. \\
  \hline  
  $3$ & \includegraphics[scale=0.25]{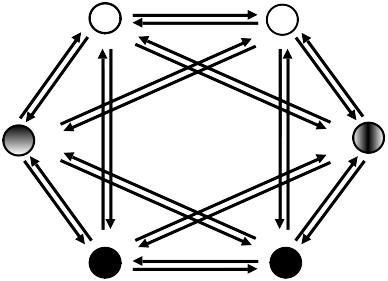}   & $(7) \; \;\{x_1=x_2, x_4=x_5\}$  \newline $(8) \; \;\{x_1=x_3,x_4=x_6\}$   \newline $(9) \; \;
  \{x_1=x_5,x_2=x_4\}$ \newline $(10) \; \;
  \{x_1=x_6,x_3=x_4\} $  \newline $(11) \; \;
  \{x_2=x_3,x_5=x_6\} $  \newline $(12) \; \;
  \{x_2=x_6,x_3=x_5\} $ & Reflection  to the plane through opposite cells and opposite edges. \\
 \hline
 $4$ & \includegraphics[scale=0.25]{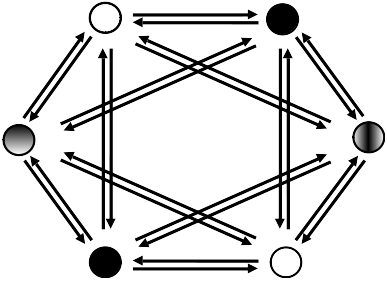}   &  $(13) \; \;\{x_1=x_4,x_2=x_5\}$   \newline $(14) \; \;
  \{x_1=x_4,x_3=x_6\}  $ \newline $(15) \; \;
  \{x_2=x_5,x_3=x_6\}$    &  Rotation of {$\pi$} of the octahedron w.r.t 
the line through opposite cells. \\
 \hline
 $5$ & \includegraphics[scale=0.25]{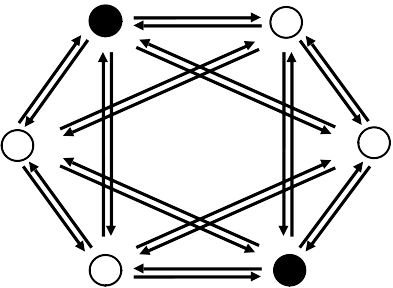}   & $(16) \; \;\{x_1=x_2=x_4=x_5,x_3=x_6\} $  \newline $(17) \; \; \{x_1=x_3=x_4=x_6,x_2=x_5\} $  \newline $(18) \; \; \{x_1=x_4,x_2=x_3=x_5=x_6\} $    & Rotation of {$\pi /2 $} w.r.t. an axis through opposite cells and a reflection w.r.t. the perpendicular plane. \\
 \hline
 $6$ & \includegraphics[scale=0.25]{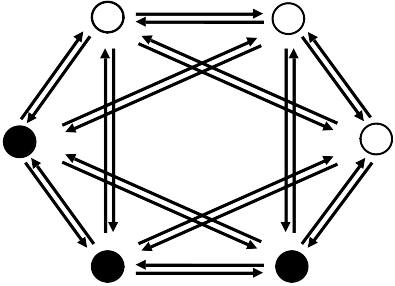}   & $(19) \; \;\{x_1=x_2=x_3,x_4=x_5=x_6\}$   \newline $(20) \; \;\{x_1=x_2=x_6,x_3=x_4=x_5\} $   \newline $(21) \; \;\{x_1=x_3=x_5,x_2=x_4=x_6\}  $ \newline $(22) \; \;\{x_1=x_5=x_6,x_2=x_3=x_4\} $  & Rotation of {$2\pi/3$} w.r.t. the line that cuts the opposite faces. \\
 \hline  
 $7$ & \includegraphics[scale=0.25]{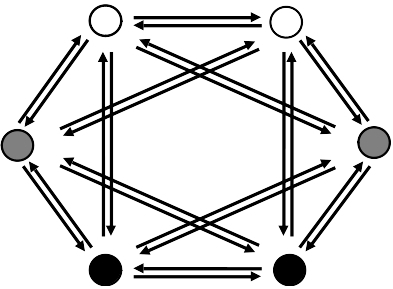}   & $(23) \; \;\{x_1=x_2,x_3=x_6,x_4=x_5\} $   \newline $(24) \; \;\{x_1=x_3,x_2=x_5,x_4=x_6\} $ \newline $(25) \; \;
  \{x_1=x_4,x_2=x_3,x_5=x_6\}  $ \newline $(26) \; \;
  \{x_1=x_4,x_2=x_6,x_3=x_5\} $   \newline $(27) \; \;
  \{x_1=x_5,x_2=x_4,x_3=x_6\} $ \newline $(28) \; \;
  \{x_1=x_6,x_2=x_5,x_3=x_4\} $     & Reflection w.r.t. two planes, one through two opposite cells and the other orthogonally through four cells. \\
 \hline

 $8$ & \includegraphics[scale=0.25]{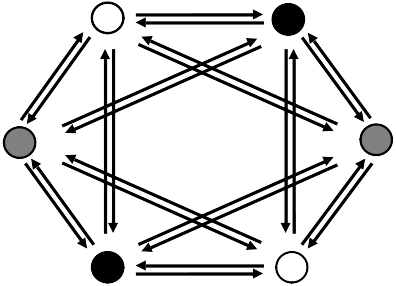}   & $(29) \; \; \{x_1=x_4,x_2=x_5,x_3=x_6\}$     & Rotation of {$\pi $} w.r.t. an axis through opposite cells and a reflection w.r.t. the perpendicular plane. \\
 \hline
\end{tabular}}} 
\caption{ Synchrony patterns of the network graph $G_6$ together with their symmetries. }
\end{table}

\begin{table}[H]
\centering 
\begin{tabular}{ |m{0.3cm}|m{4cm}|m{5cm}|m{1cm}|} 
\hline
 $\sharp$ & \quad  \quad \quad  \quad \quad $\CG$ & Synchrony patterns (numbering extracted from Table~1) &  $| \Sigma(\CG) | $ \\
\hline
$1$ & {\includegraphics[scale=0.25]{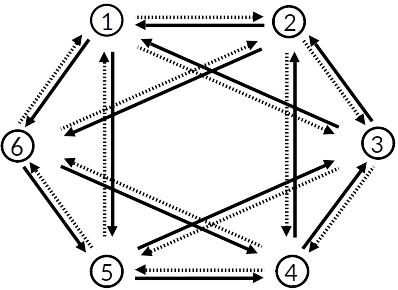}} & $16,17,18,21,29$ & $1$ \\
\hline 
$2$ & {\includegraphics[scale=0.25]{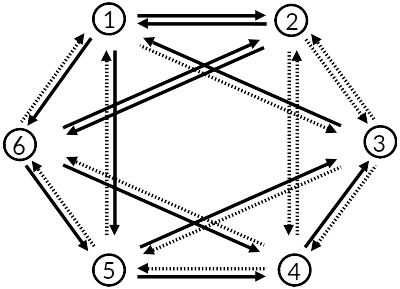}} & $2,16,17,18,21, 29$ & $6$ \\ \hline  
$3$ & {\includegraphics[scale=0.25]{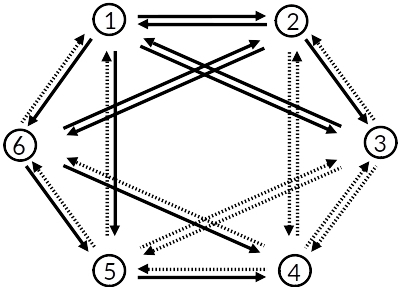}} & $2,3,6,12,15,16,17,18,21, \newline26,29$  & $6$ \\ \hline
$4$ & {\includegraphics[scale=0.25]{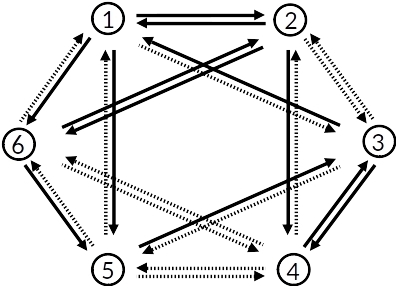}} & $1,2,13,16,17,18,21,  29$   & $6$ \\ \hline 
$5$ & {\includegraphics[scale=0.25]{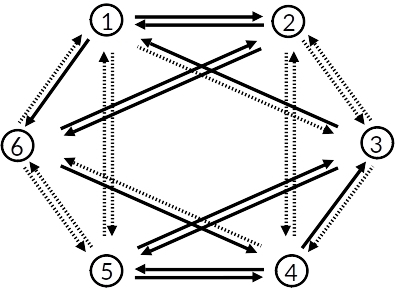}} & $9,12,16,17,18, 21, 26,27,  29$   & $3$ \\ \hline
$6$ & {\includegraphics[scale=0.25]{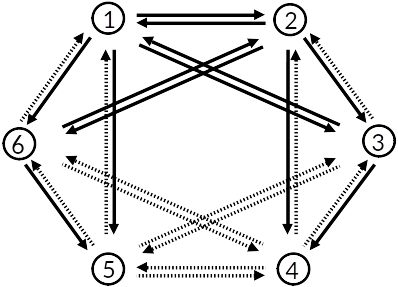}}  & $1,2,3,5,6,8, 12,13,14,15,16, \newline 17, 18,21,23,24,26,29$  \newline  & $3$ \\ \hline 
$7$ & {\includegraphics[scale=0.25]{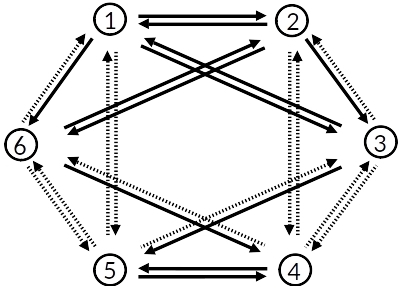}} & $3,9,16,17,18,21,27, 29$   \newline & $6$ \\ \hline 
$8$ & {\includegraphics[scale=0.25]{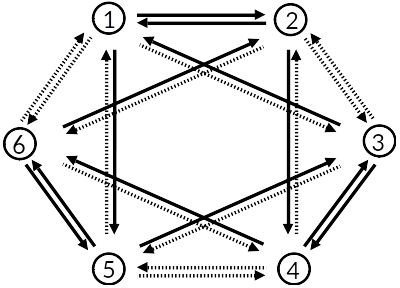}} & $1,2,3,13,14,15,16,17,18,21,\newline 
23,25,28,29 $  \newline  & $1$ \\ \hline 
\end{tabular}
\caption{Admissible network graphs for \eqref{eq:vf on R6} up to ODE-equivalence with  their corresponding synchrony patterns. The last column gives the number of admissible graphs nonisomorphic and ODE-equivalent  to the graph.}
\end{table}

\bibliographystyle{alpha}

\vspace*{1cm}



\begin{thebibliography}{100}

\bibitem{AS}Abrams, D. M. and Strogatz, S. H., Chimera states in a ring of nonlocally coupled oscillators. {\it Int. J. Bif.
Chaos} {\bf 16} (2006) 21--37.

\bibitem{ABCR} Acebr\'on, J.A., Bonilla, L.L., P\'erez-Vicente, C.J., Ritort, F., Spigler, R., The Kuramoto model: a simple paradigm for synchronization. {\it Rev. Mod. Phys.} {\bf 77}(1) (2005) 137--185.

 \bibitem{AguiarDiasAlg} Aguiar,  M. A. D. and Dias, A. P. S., The lattice of synchrony subspaces of a coupled cell network: Characterization and computation algorithm. {\it J. Non. Sci.}, {\bf 24}(6) (2014) 949--996.

 \bibitem{AntoneliStewart} Antoneli, F. and Stewart, I., Symmetry and synchrony in coupled cell networks 1: Fixed-point spaces. {\it Int. J. Bif. Chaos} {\bf 16}(3)  (2006) 559--577.
 
\bibitem{Braci1990} Bracikowski, C. and R. Roy, Chaos in a multimode solid-state laser system {\it Chaos} {\bf 1} (1990) 49--64.

\bibitem{CHANDRASEKAR} Chandrasekar,V. K.,  Gopal, R., Venkatesan, A. and Lakshmanan, M., Mechanism for intensity-induced chimera states in globally coupled oscillators. {\it Phys. Rev. E} {\bf 90}(6) (2014) p. 062913.


 \bibitem{StewartCohen2000} Cohen, J., Stewart, I., and Elmhirst, T., Symmetry-Breaking as an Origin of Species. In: Bifurcation, Symmetry and Patterns, Trends in Mathematics, Birk\"auser (2003) 3--54.

\bibitem{DiasStewart}  Dias,  A. P. S. and Stewart, I., Linear equivalence and ODE-equivalence for coupled cell networks {\it Nonlinearity} {\bf 18} (2005) 1003--20 .

 \bibitem{Fowles1986} Fowles, G. R., {\it Analytical Mechanics} Saunders Col. Pub., Philadelphia (1986).
	
\bibitem{Griffiths1985} Griffiths J. B., {\it The Theory of Classical Dynamics}, Camb. Univ. Press, Cambridge (1985). 
	
\bibitem{GNS}Golubitsky, M., Nicol, M. and Stewart, I., Some curious phenomena in coupled cell networks, {\it J. Non. Sci.}  {\bf 14} (2004) 207--236.

\bibitem{SM2002} Golubitsky, M., Stewart, I., {\it The Symmetry Perspective: From Equilibrium to Chaos in Phase
	Space and Physical Space}, Progr. Math. 200, Birkh\"auser Verlag, Basel (2002).
	
\bibitem{SMT} Golubitsky, M., Stewart, I., and T\"or\"ok, A., Patterns of synchrony in coupled cell networks with multiple arrows,  {\it SIAM J. Appl. Dyn. Syst.} {\bf 4}(1) (2005)	78--100.
	
\bibitem{[1988]} Hadley, P., M. R. Beasley, and K. Wiesenfeld, Phase locking of Josephson-junction series arrays, {\it Phys. Rev. B} {\bf 38}(13) (1988) 8712--8719.

 \bibitem{Kopell1988} Kopell, N. and Ermentrout, G. B., Coupled oscillators and the design of central pattern generators, {\it Math. Biosci.} {\bf 89} (1988) 14--23 .
 
\bibitem{Kopell1990} Kopell, N. and Ermentrout, G. B., Phase transitions and other phenomena in chains of oscillators, {\it SIAM J. Appl. Math.} {\bf 50}  (1990) 1014--1052.

 \bibitem{Kopell1986} Kopell, N. and  Ermentrout, G. B., Symmetry and phaselocking in chains of weakly coupled oscillators, {\it  Comm. Pure Appl. Math.} {\bf 39} (1986)  623--660.
	
\bibitem{KB}Kuramoto, Y. and Battogtokh, D.,  Coexistence of coherence and incoherence in nonlocally coupled phase
oscillators {\it  Non. Phenom. Compl. Syst} {\bf 5}(2002)  
380--5 .

\bibitem{MTFH} Martens, E.A.,  Thutupalli, S.,  Fourri\`ere, A.  and Hallatschek, O., Chimera states in mechanical oscillator networks, {\it Proc. Nat. Ac. Sci. USA} {\bf 110}(26) (2013) 10563--10567.

\bibitem{PA} Panaggio, M.J.  and Abrams, D.M., Chimera states: coexistence of coherence and
incoherence in networks of coupled oscillators. {\it  Nonlinearity} {\bf 28}(3)  (2015) R67.

\bibitem{Cerdeira} Simo, G.R.,  Njougouo, T.,  Aristides, R. P., Louodop, P., Tchitnga, R. and Cerdeira, H.A., Chimera states in a neuronal network under the action of an electric field, {\it Phys. Rev. E} {\bf 103}(6)  (2021) 623041-11. 

\bibitem{SMP} Stewart, I., Golubitsky, M. and Pivato, M. Symmetry groupoids and patterns of synchrony in coupled cell networks, {\it SIAM J. Appl. Dyn. Syst} {\bf 2}(4)  (2003) 609--646.

	
	
\bibitem{WW1988} Wang, S. S. and Winful, H. G., Dynamics of phase-locked semiconductor laser arrays, {\it Appl. Phys. Let.} {\bf 52}  (1988) 1744--1776.

\bibitem{Anna} Zakharova, A., {\it Chimera Patterns in Networks}, Springer: Complexity (2020).
	
\end{thebibliography}

\end{document}